\documentclass[11pt]{amsart}    
\usepackage{latexsym, amsmath, amsthm, amssymb, setspace,verbatim}
\usepackage[all]{xy}
\usepackage{longtable}
\usepackage[utf8]{inputenc} \usepackage[T1]{fontenc} \usepackage{lmodern}
\usepackage{hyperref}
\usepackage{enumitem}

\title[Rokhlin property for compact quantum groups]{ The spatial Rokhlin property for actions of compact quantum groups }
\author{ Sel\c{c}uk Barlak }
\address{Department of Mathematics and Computer Science \\
         University of Southern Denmark \\
         Campusvej 55 \\
         DK-5230 Odense M\\
         Denmark 
}
\email{barlak@imada.sdu.dk}
\author{ G\'{a}bor Szab\'{o} }
\address{Fraser Noble Building \\
         Institute of Mathematics \\ 
         University of Aberdeen \\ 
         Aberdeen AB24 3UE \\
         United Kingdom   
}
\email{ gabor.szabo@abdn.ac.uk }

\author{Christian Voigt}
\address{School of Mathematics and Statistics \\
         University of Glasgow \\
         15 University Gardens \\
         Glasgow G12 8QW \\
         United Kingdom 
}
\email{christian.voigt@glasgow.ac.uk}


\numberwithin{equation}{section}

\begin{document}

\renewcommand\matrix[1]{\left(\begin{array}{*{10}{c}} #1 \end{array}\right)}  
\newcommand\set[1]{\left\{#1\right\}}  
\newcommand\mset[1]{\left\{\!\!\left\{#1\right\}\!\!\right\}}

\newcommand{\IA}[0]{\mathbb{A}} \newcommand{\IB}[0]{\mathbb{B}}
\newcommand{\IC}[0]{\mathbb{C}} \newcommand{\ID}[0]{\mathbb{D}}
\newcommand{\IE}[0]{\mathbb{E}} \newcommand{\IF}[0]{\mathbb{F}}
\newcommand{\IG}[0]{\mathbb{G}} \newcommand{\IH}[0]{\mathbb{H}}
\newcommand{\II}[0]{\mathbb{I}} \renewcommand{\IJ}[0]{\mathbb{J}}
\newcommand{\IK}[0]{\mathbb{K}} \newcommand{\IL}[0]{\mathbb{L}}
\newcommand{\IM}[0]{\mathbb{M}} \newcommand{\IN}[0]{\mathbb{N}}
\newcommand{\IO}[0]{\mathbb{O}} \newcommand{\IP}[0]{\mathbb{P}}
\newcommand{\IQ}[0]{\mathbb{Q}} \newcommand{\IR}[0]{\mathbb{R}}
\newcommand{\IS}[0]{\mathbb{S}} \newcommand{\IT}[0]{\mathbb{T}}
\newcommand{\IU}[0]{\mathbb{U}} \newcommand{\IV}[0]{\mathbb{V}}
\newcommand{\IW}[0]{\mathbb{W}} \newcommand{\IX}[0]{\mathbb{X}}
\newcommand{\IY}[0]{\mathbb{Y}} \newcommand{\IZ}[0]{\mathbb{Z}}

\newcommand{\CA}[0]{\mathcal{A}} \newcommand{\CB}[0]{\mathcal{B}}
\newcommand{\CC}[0]{\mathcal{C}} \newcommand{\CD}[0]{\mathcal{D}}
\newcommand{\CE}[0]{\mathcal{E}} \newcommand{\CF}[0]{\mathcal{F}}
\newcommand{\CG}[0]{\mathcal{G}} \newcommand{\CH}[0]{\mathcal{H}}
\newcommand{\CI}[0]{\mathcal{I}} \newcommand{\CJ}[0]{\mathcal{J}}
\newcommand{\CK}[0]{\mathcal{K}} \newcommand{\CL}[0]{\mathcal{L}}
\newcommand{\CM}[0]{\mathcal{M}} \newcommand{\CN}[0]{\mathcal{N}}
\newcommand{\CO}[0]{\mathcal{O}} \newcommand{\CP}[0]{\mathcal{P}}
\newcommand{\CQ}[0]{\mathcal{Q}} \newcommand{\CR}[0]{\mathcal{R}}
\newcommand{\CS}[0]{\mathcal{S}} \newcommand{\CT}[0]{\mathcal{T}}
\newcommand{\CU}[0]{\mathcal{U}} \newcommand{\CV}[0]{\mathcal{V}}
\newcommand{\CW}[0]{\mathcal{W}} \newcommand{\CX}[0]{\mathcal{X}}
\newcommand{\CY}[0]{\mathcal{Y}} \newcommand{\CZ}[0]{\mathcal{Z}}

\newcommand{\FA}[0]{\mathfrak{A}} \newcommand{\FB}[0]{\mathfrak{B}}
\newcommand{\FC}[0]{\mathfrak{C}} \newcommand{\FD}[0]{\mathfrak{D}}
\newcommand{\FE}[0]{\mathfrak{E}} \newcommand{\FF}[0]{\mathfrak{F}}
\newcommand{\FG}[0]{\mathfrak{G}} \newcommand{\FH}[0]{\mathfrak{H}}
\newcommand{\FI}[0]{\mathfrak{I}} \newcommand{\FJ}[0]{\mathfrak{J}}
\newcommand{\FK}[0]{\mathfrak{K}} \newcommand{\FL}[0]{\mathfrak{L}}
\newcommand{\FM}[0]{\mathfrak{M}} \newcommand{\FN}[0]{\mathfrak{N}}
\newcommand{\FO}[0]{\mathfrak{O}} \newcommand{\FP}[0]{\mathfrak{P}}
\newcommand{\FQ}[0]{\mathfrak{Q}} \newcommand{\FR}[0]{\mathfrak{R}}
\newcommand{\FS}[0]{\mathfrak{S}} \newcommand{\FT}[0]{\mathfrak{T}}
\newcommand{\FU}[0]{\mathfrak{U}} \newcommand{\FV}[0]{\mathfrak{V}}
\newcommand{\FW}[0]{\mathfrak{W}} \newcommand{\FX}[0]{\mathfrak{X}}
\newcommand{\FY}[0]{\mathfrak{Y}} \newcommand{\FZ}[0]{\mathfrak{Z}}

\newcommand{\Fa}[0]{\mathfrak{a}} \newcommand{\Fb}[0]{\mathfrak{b}}
\newcommand{\Fc}[0]{\mathfrak{c}} \newcommand{\Fd}[0]{\mathfrak{d}}
\newcommand{\Fe}[0]{\mathfrak{e}} \newcommand{\Ff}[0]{\mathfrak{f}}
\newcommand{\Fg}[0]{\mathfrak{g}} \newcommand{\Fh}[0]{\mathfrak{h}}
\newcommand{\Fi}[0]{\mathfrak{i}} \newcommand{\Fj}[0]{\mathfrak{j}}
\newcommand{\Fk}[0]{\mathfrak{k}} \newcommand{\Fl}[0]{\mathfrak{l}}
\newcommand{\Fm}[0]{\mathfrak{m}} \newcommand{\Fn}[0]{\mathfrak{n}}
\newcommand{\Fo}[0]{\mathfrak{o}} \newcommand{\Fp}[0]{\mathfrak{p}}
\newcommand{\Fq}[0]{\mathfrak{q}} \newcommand{\Fr}[0]{\mathfrak{r}}
\newcommand{\Fs}[0]{\mathfrak{s}} \newcommand{\Ft}[0]{\mathfrak{t}}
\newcommand{\Fu}[0]{\mathfrak{u}} \newcommand{\Fv}[0]{\mathfrak{v}}
\newcommand{\Fw}[0]{\mathfrak{w}} \newcommand{\Fx}[0]{\mathfrak{x}}
\newcommand{\Fy}[0]{\mathfrak{y}} \newcommand{\Fz}[0]{\mathfrak{z}}

\newcommand{\Ra}[0]{\Rightarrow}
\newcommand{\La}[0]{\Leftarrow}
\newcommand{\LRa}[0]{\Leftrightarrow}

\renewcommand{\phi}[0]{\varphi}
\newcommand{\eps}[0]{\varepsilon}

\newcommand{\quer}[0]{\overline}
\newcommand{\uber}[0]{\choose}
\newcommand{\ord}[0]{\operatorname{ord}}		
\newcommand{\GL}[0]{\operatorname{GL}}
\newcommand{\supp}[0]{\operatorname{supp}}	
\newcommand{\id}[0]{\operatorname{id}}		
\newcommand{\Sp}[0]{\operatorname{Sp}}		
\newcommand{\eins}[0]{\mathbf{1}}			
\newcommand{\diag}[0]{\operatorname{diag}}
\newcommand{\ind}[0]{\operatorname{ind}}
\newcommand{\auf}[1]{\quad\stackrel{#1}{\longrightarrow}\quad}
\newcommand{\hull}[0]{\operatorname{hull}}
\newcommand{\prim}[0]{\operatorname{Prim}}
\newcommand{\ad}[0]{\operatorname{Ad}}
\newcommand{\quot}[0]{\operatorname{Quot}}
\newcommand{\ext}[0]{\operatorname{Ext}}
\newcommand{\ev}[0]{\operatorname{ev}}
\newcommand{\fin}[0]{{\subset\!\!\!\subset}}
\newcommand{\diam}[0]{\operatorname{diam}}
\newcommand{\Hom}[0]{\operatorname{Hom}}
\newcommand{\Aut}[0]{\operatorname{Aut}}
\newcommand{\del}[0]{\partial}
\newcommand{\dimeins}[0]{\dim^{\!+1}}
\newcommand{\dimnuc}[0]{\dim_{\mathrm{nuc}}}
\newcommand{\dimnuceins}[0]{\dimnuc^{\!+1}}
\newcommand{\dr}[0]{\operatorname{dr}}
\newcommand{\dimrok}[0]{\dim_{\mathrm{Rok}}}
\newcommand{\dimrokeins}[0]{\dimrok^{\!+1}}
\newcommand{\dreins}[0]{\dr^{\!+1}}
\newcommand*\onto{\ensuremath{\joinrel\relbar\joinrel\twoheadrightarrow}} 
\newcommand*\into{\ensuremath{\lhook\joinrel\relbar\joinrel\rightarrow}}  
\newcommand{\im}[0]{\operatorname{im}}
\newcommand{\dst}[0]{\displaystyle}
\newcommand{\cstar}[0]{$\mathrm{C}^*$}
\newcommand{\ann}[0]{\operatorname{Ann}}
\newcommand{\dist}[0]{\operatorname{dist}}
\newcommand{\asdim}[0]{\operatorname{asdim}}
\newcommand{\asdimeins}[0]{\operatorname{asdim}^{\!+1}}
\newcommand{\amdim}[0]{\dim_{\mathrm{am}}}
\newcommand{\amdimeins}[0]{\amdim^{\!+1}}
\newcommand{\dimrokc}[0]{\dim_{\mathrm{Rok,c}}}
\newcommand{\dimrokceins}[0]{\dimrokc^{\!+1}}
\newcommand{\act}[0]{\operatorname{Act}}
\newcommand{\idlat}[0]{\operatorname{IdLat}}
\newcommand{\Cu}[0]{\operatorname{Cu}}
\newcommand{\Ost}[0]{\CO_\infty^{\mathrm{st}}}
\newcommand{\ue}[1]{{~\approx_{\mathrm{u},#1}}~}
\newcommand{\ueo}[0]{\approx_{\mathrm{u}}}

\newcommand*{\ket}{\rangle}
\newcommand*{\bra}{\langle}
\newcommand*{\red}{\mathsf{r}}
\renewcommand*{\max}{\mathsf{f}}
\newcommand*{\opp}{\mathsf{opp}}
\newcommand*{\cop}{\mathsf{cop}}
\newcommand*{\f}{\mathsf{f}}
\newcommand*{\yd}{\mathsf{YD}}
\newcommand*{\DD}{\mathsf{D}}
\newcommand*{\Rep}{\mathsf{Rep}}
\newcommand*{\Corep}{\mathsf{Corep}}
\newcommand*{\Irr}{\mathsf{Irr}}

\newtheorem{satz}{Satz}[section]		
\newtheorem{cor}[satz]{Corollary}
\newtheorem{lemma}[satz]{Lemma}
\newtheorem{prop}[satz]{Proposition}
\newtheorem{theorem}[satz]{Theorem}
\newtheorem*{theoreme}{Theorem}

\theoremstyle{definition}
\newtheorem{defi}[satz]{Definition}
\newtheorem*{defie}{Definition}
\newtheorem{defprop}[satz]{Definition \& Proposition}
\newtheorem{nota}[satz]{Notation}
\newtheorem*{notae}{Notation}
\newtheorem{rem}[satz]{Remark}
\newtheorem*{reme}{Remark}
\newtheorem{example}[satz]{Example}
\newtheorem{defnot}[satz]{Definition \& Notation}
\newtheorem{question}[satz]{Question}
\newtheorem*{questione}{Question}
\newtheorem{contruction}[satz]{Contruction}

\newenvironment{bew}{\begin{proof}[Proof]}{\end{proof}}

\begin{abstract}
We introduce the spatial Rokhlin property for actions of coexact compact quantum groups on \cstar-algebras, generalizing the Rokhlin 
property for both actions of classical compact groups and finite quantum groups. Two key ingredients in our approach are 
the concept of sequentially split $*$-homomorphisms, and the use of braided tensor products instead of ordinary tensor products. 

We show that various structure results carry over from the classical theory to this more general setting. 
In particular, we show that a number of \cstar-algebraic properties relevant to the classification program pass from the 
underlying \cstar-algebra of a Rokhlin action to both the crossed product and the fixed point algebra. 
Towards establishing a classification theory, we show that Rokhlin actions exhibit a rigidity property with respect to approximate 
unitary equivalence. 
Regarding duality theory, we introduce the notion of spatial approximate representability for actions of discrete quantum groups. 
The spatial Rokhlin property for actions of a coexact compact quantum group is shown to be dual to spatial approximate representability 
for actions of its dual discrete quantum group, and vice versa.
\end{abstract}

\maketitle


\tableofcontents

\setcounter{section}{-1}

\section{Introduction}

\noindent
The Rokhlin property for finite group actions on unital \cstar-algebras was introduced and studied by Izumi in \cite{Izumi04, Izumi04II}, building on earlier work of Herman-Jones \cite{HermanJones82} and Herman-Ocneanu \cite{HermanOcneanu84}. 
Since the very beginning it has proven to be a useful tool in the theory of finite group actions. The Rokhlin property was subsequently generalized by Hirshberg-Winter \cite{HirshbergWinter07} to the case of compact groups, and studied further by Gardella \cite{Gardella14}; see also \cite{GardellaSantiago15, GardellaSantiago16}. For finite quantum group actions, Kodaka-Teruya  introduced and studied the Rokhlin property and approximate representability in \cite{KodakaTeruya15}.

The established theory, which we shall now briefly summarize in the initial setting of finite group actions, has three particularly 
remarkable features:
The first is a multitude of permanence properties; it is known that many \cstar-algebraic properties pass from the 
coefficient \cstar-algebra to the crossed product and fixed point algebra. This was in part addressed by Izumi in \cite{Izumi04}, and 
studied more in depth by Osaka-Phillips \cite{OsakaPhillips12} and Santiago \cite{Santiago15}.
The second feature is rigidity with respect to approximate unitary equivalence; a result of Izumi \cite{Izumi04} asserts that two Rokhlin actions of a finite group on a separable, unital \cstar-algebra are conjugate via an approximately inner automorphism if and only if the two actions are pointwise approximately unitarily equivalent; see \cite{GardellaSantiago15} for the non-unital case and \cite{GardellaSantiago16} for the case of compact groups. As demonstrated by Izumi in \cite{Izumi04}, 
this rigidity property is useful for classifying Rokhlin actions on classifiable \cstar-algebras via $K$-theoretic invariants. 
The third feature is duality theory; a result of Izumi \cite{Izumi04} shows that an action of a finite abelian group on a separable, 
unital \cstar-algebra has the Rokhlin property if and only if the dual action is approximately representable, and vice versa. This
has been generalized to the non-unital case by Nawata \cite{Nawata16}, and to actions of compact abelian groups by the first two 
authors \cite{BarlakSzabo15}; see also \cite{GardellaPHD}.

In the present paper, we introduce and study the spatial Rokhlin property for actions of coexact compact quantum groups, generalising 
and unifying the work mentioned above. In particular, we carry over various structure results from the classical to the general 
case. Firstly, this allows us to remove all commutativity assumptions in the study of duality properties for Rokhlin actions. This 
is relevant even for classical group actions. Indeed, the Pontrjagin dual of a nonabelian group is no longer a group, but can be viewed 
as a quantum group. Accordingly, a natural way to fully incorporate nonabelian groups into the picture is to work in the setting of quantum groups from the very beginning. Secondly, it turns out that some results can be given quite short and transparent proofs in this more 
abstract setup, simpler than in previous accounts. Finally, our results are also of interest from 
the point of view of quantum group theory. Indeed, they provide examples of quantum group actions that either allow for classification, or 
the systematical analysis of structural properties of crossed product \cstar-algebras, in particular whether they fall within the scope of 
the Elliott program. 

Let us highlight two comparably new ingredients in our approach. The first is the notion of (equivariantly) sequentially 
split $*$-homomorphisms introduced by the first two authors in \cite{BarlakSzabo15}. It has already been demonstrated in \cite{BarlakSzabo15} that many structural results related to the Rokhlin property can be recast and conceptually proved in the language of sequentially split $*$-homomorphisms, and some new ones could be proved as well. The second ingredient is a purely quantum feature, namely the braided tensor product construction. This provides the correct substitute for tensor product actions in the classical theory, and it also gives a conceptual 
explanation of the fact that the central sequence algebra is no longer the right tool in the quantum setting. 
Being widely known in the algebraic theory of quantum groups, braided tensor products in the operator algebraic framework were 
first introduced and studied in \cite{NestVoigt2010}.

As already indicated above, for most of the paper we will assume that our quantum groups satisfy exactness/coexactness assumptions. 
This may appear surprising at first sight; it is essentially due to the fact that we have chosen to work in a reduced setting, that is, with 
reduced crossed products and minimal (braided) tensor products. 
We shall indicate at several points in the main text where precisely exactness enters. Our setup 
yields the strongest versions of conceivable definitions of Rokhlin actions and approximately 
representable actions, however, at the same time our examples are restricted to the amenable/coamenable case. 
On the other hand, the reduced setting matches best with the existing literature on quantum group actions, see for instance 
\cite{BaajSkandalis93}, \cite{BSV2003}, \cite{Timmermann08}, \cite{NestVoigt2010} and references therein.  
The necessary modifications to set up a full version of our theory are mainly of technical nature; we have refrained from 
carrying this out here.

Let us now explain how the paper is organized. In Section 1, we gather some preliminaries and background on quantum groups, 
including a review of Takesaki-Takai duality and braided tensor products. Section 2 deals with induced actions of discrete and compact 
quantum groups on sequence algebras. Already for classical compact groups
these actions typically fail to be continuous, and in the quantum setting this leads to a number of subtle issues. 
In Section 3, we define and study equivariantly sequentially split $*$-homomorphisms. We show that, as in the case of group actions, this notion behaves well with respect to crossed products and fixed point algebras. We also establish a general duality result for equivariantly sequentially split $*$-homomorphisms. In Section 4, we introduce the spatial Rokhlin property for actions of coexact compact quantum groups, 
and spatial approximate representability for actions of exact discrete quantum groups. We verify that various \cstar-algebraic properties pass 
to crossed products and fixed point algebras. Moreover, we show that the spatial Rokhlin property and spatial approximate representability 
are dual to each other. In Section 5, we present some steps towards a classification theory for actions with the spatial Rokhlin property. Among other things, we prove that two such $G$-actions on a \cstar-algebra $A$ are conjugate via an approximately inner automorphism if and only if the actions are approximately unitarily equivalent as $*$-homomorphisms from $A$ to $C^\red(G)\otimes A$. This generalizes a number
of previous such classification results, in particular those of Izumi \cite{Izumi04}, Gardella-Santiago \cite{GardellaSantiago16} and Kodaka-Teruya \cite{KodakaTeruya15}. 
In this section, we also generalize a $K$-theory formula for the fixed algebra of a Rokhlin action, first proved for certain finite group cases by Izumi \cite{Izumi04} and for compact group actions by the first two authors in \cite{BarlakSzabo15}. 
Finally, in Section 6 we present some examples of Rokhlin actions. In particular, we show that any coamenable compact quantum group
admits an essentially unique action with the spatial Rokhlin property on $ \CO_2 $.

The work presented here was initiated while the first two authors participated in the conference CSTAR at the University of Glasgow in September 2014. Substantial parts of this work were carried out during research visits of the authors to Oberwolfach in June 2015, of the first two authors at the Mittag-Leffler Institute from January to March 2016, and of the third author to the University of Southern Denmark in April 2016. The authors are grateful to all these institutions for their hospitality and support. 

The first author was supported by GIF Grant 1137-30.6/2011, ERC AdG 267079, SFB 878 `Groups, Geometry, and Actions' and the Villum Fonden project grant `Local and global structures of groups and their algebras' (2014-2018). 
The second author was supported by SFB 878 `Groups, Geometry, and Actions' and the Engineering and Physical Sciences Research Council Grant EP/N00874X/1 `Regularity and dimension for \cstar-algebras'. 
The third author was supported by the Engineering and Physical Sciences Research Council Grant
EP/L013916/1 and the Polish National Science Centre grant no.\ 2012/06/M/ST1/00169.

The authors would like to thank the referee for a thorough reading of the manuscript and for helpful suggestions.


\section{Preliminaries} \label{secqg}

In this preliminary section we collect some definitions and results 
from the theory of quantum groups and 
fix our notation. We will mainly follow the conventions in \cite{NestVoigt2010} as far as general quantum group theory is concerned. 
For more detailed information and background we refer to \cite{Woronowicz98}, \cite{MaesvanDaele98}, 
\cite{KustermansVaes2000}, \cite{KVVDW01}. 

Let us make some general remarks on the notation used throughout the paper.  We write $ \IL(\CE) $ for the space of adjointable
operators on a Hilbert $ A $-module, and $ \IK(\CE) $ denotes the space of compact operators. 
The closed linear span of a subset $ X $ of a Banach space is denoted by $ [X] $. 
If $ x,y $ are elements of a Banach space and $ \eps > 0 $ we write $ x =_\eps y $ if $ \|x - y \| < \eps $. 

Depending on the context, the symbol
$ \otimes $ denotes either the tensor product of Hilbert spaces, the minimal tensor product of \cstar-algebras, or the tensor product
of von Neumann algebras. We write $ \odot $ for algebraic tensor products. 

If $ A $ and $ B $ are \cstar-algebras then the flip map $ A \otimes B \rightarrow B \otimes A $ is denoted by $ \sigma $. 
That is, we have $ \sigma(a \otimes b) = b \otimes a $ for $ a \in A, b \in B $. 

If $ \CH $ is a Hilbert space we write $ \Sigma \in \IL(\CH \otimes \CH) $ for the flip map $ \Sigma (\xi \otimes \eta) = \eta \otimes \xi $. 
For operators on multiple tensor products we use the leg numbering notation. For instance, if $ W \in \IL(\CH \otimes \CH) $ 
is an operator on $ \CH \otimes \CH $, then $ W_{12} = W \otimes \id \in \IL(\CH \otimes \CH \otimes \CH) $ and 
$ W_{23} = \id \otimes W $. Moreover, $ W_{13} = \Sigma_{12} W_{23} \Sigma_{12} $. 

If $B$ is a \cstar-algebra we write $ \tilde{B} $ for the smallest unitarization of $B$.

\subsection{Quantum groups} 

Although we will only be interested in compact and discrete quantum groups, let us first recall a few 
defininitions and facts regarding general locally compact quantum groups. 

Let $ \phi $ be a normal, semifinite and faithful weight on a von Neumann algebra $ M $. We use the standard notation
\[
\mathcal{M}^+_\phi = \{ x \in M_+ \mid \phi(x) < \infty \}, \qquad \mathcal{N}_\phi = \{ x \in M \mid \phi(x^* x) < \infty \}
\]
and write $ M_*^+ $ for the space of positive normal linear functionals on $ M $. Assume that
$ \Delta: M \rightarrow M \otimes M $ is a normal unital $ * $-homomorphism. The weight $ \phi $ is called left invariant
with respect to $ \Delta $ if
\[
\phi((\omega \otimes \id)\Delta(x)) = \phi(x) \omega(\eins)
\]
for all $ x \in \mathcal{M}_\phi^+ $ and $ \omega \in M_*^+ $. Similarly one defines the notion of a right invariant weight.
\begin{defi} \label{defqg}
A locally compact quantum group $ G $ is given by a von Neumann algebra $ L^\infty(G) $ together with a normal unital $ * $-homomorphism
$ \Delta: L^\infty(G) \rightarrow L^\infty(G) \otimes L^\infty(G) $ satisfying the coassociativity relation
\[
(\Delta \otimes \id)\Delta = (\id \otimes \Delta)\Delta
\]
and normal semifinite faithful weights $ \phi $ and $ \psi $ on $ L^\infty(G) $ which are left and right invariant, respectively.
The weights $ \phi $ and $ \psi $ will also be referred to as Haar weights of $ G $. 
\end{defi}

\begin{rem} 
Our notation should help to make clear how ordinary locally compact groups can be viewed as quantum groups.
Indeed, if $ G $ is a locally compact group, then the algebra $ L^\infty(G) $ of essentially bounded measurable functions on $ G $ together with the comultiplication
$ \Delta: L^\infty(G) \rightarrow L^\infty(G) \otimes L^\infty(G) $ given by
\[
\Delta(f)(s,t) = f(st)
\]
defines a locally compact quantum group. The weights $ \phi $ and $ \psi $ are given in this case by left and right Haar measures, respectively. Of course, for a general locally compact quantum group $ G $ the notation $ L^\infty(G) $ is purely formal.
Similar remarks apply to the \cstar-algebras $ C^*_\max(G), C^*_\red(G) $ and $ C^\max_0(G), C^\red_0(G) $ associated to $ G $ that we discuss below. It is convenient to view all of them as different appearances of the quantum group $ G $. 
\end{rem} 

\begin{rem}[cf.\ \cite{KustermansVaes2000}] \label{Wremark}
Let $ G $ be a locally compact quantum group and let $ \Lambda: \mathcal{N}_\phi \rightarrow L^2(G) $ be the GNS-construction for the 
Haar weight $ \phi $. Throughout the paper we will only consider second countable quantum groups, that is, 
quantum groups for which $ L^2(G) $ is a separable Hilbert space. 

One obtains a unitary $ W_G = W $ on $ L^2(G) \otimes L^2(G) $ by
\[
W^*(\Lambda(x) \otimes \Lambda(y)) = (\Lambda \otimes \Lambda)(\Delta(y)(x \otimes \eins))
\]
for all $ x, y \in \mathcal{N}_\phi $. This unitary is multiplicative, which means that $ W $ satisfies the pentagonal equation
\[
W_{12} W_{13} W_{23} = W_{23} W_{12}.
\]
From $ W $ one can recover the von Neumann algebra $ L^\infty(G) $ as the strong closure of the algebra
$ (\id \otimes \IL(L^2(G))_*)(W) $, where $ \IL(L^2(G))_* $ denotes the space of normal linear functionals on $ \IL(L^2(G)) $. Moreover
one has
\[
\Delta(x) = W^*(\eins \otimes x) W
\]
for all $ x \in M $. The algebra $ L^\infty(G) $ has an antipode, which
is an unbounded, $ \sigma $-strong* closed linear map $ S $ given by $ S(\id \otimes \omega)(W) = (\id \otimes \omega)(W^*) $
for $ \omega \in \IL(L^2(G))_* $. Moreover, there is a polar decomposition $ S = R \tau_{-i/2} $ where $ R $ is an
antiautomorphism of $ L^\infty(G) $ called the unitary antipode and $ (\tau_t) $ is a strongly continuous one-parameter group
of automorphisms of $ L^\infty(G) $ called the scaling group. The unitary antipode satisfies $ \sigma \circ (R \otimes R) \circ \Delta 
= \Delta \circ R $. 

The group-von Neumann algebra $ \mathcal{L}(G) $ of the quantum group $ G $ is the strong
closure of the algebra $ (\IL(L^2(G))_* \otimes \id)(W) $ with the comultiplication $ \hat{\Delta}: \mathcal{L}(G) \rightarrow \mathcal{L}(G) \otimes \mathcal{L}(G) $
given by
\[
\hat{\Delta}(y) = \hat{W}^* (\eins \otimes y) \hat{W}
\]
where $ \hat{W} = \Sigma W^* \Sigma $ and $ \Sigma \in \IL(L^2(G) \otimes L^2(G)) $ is the flip map. It defines
a locally compact quantum group $ \hat{G} $ which is called the dual of $ G $. The left invariant weight
$ \hat{\phi} $ for the dual quantum group has a GNS-construction $ \hat{\Lambda}: \mathcal{N}_{\hat{\phi}} \rightarrow L^2(G) $,
and according to our conventions we have $ \mathcal{L}(G) = L^\infty(\hat{G}) $. 
\end{rem}

\begin{rem} 
Since we are following the conventions of Kustermans and Vaes \cite{KustermansVaes2000}, there is a flip map built into the definition 
of $ \hat{\Delta} $. As we will see below, this is a natural choice when working with Yetter-Drinfeld actions; however, it is slightly inconvenient when it comes to Takesaki-Takai duality. We will write $ \check{G} $ for the quantum group 
corresponding to $ \mathcal{L}(G)^\cop $. That is, $ L^\infty(\check{G}) $ is the von Neumann algebra $ \mathcal{L}(G) $ equipped 
with the opposite comultiplication $ \check{\Delta} = \hat{\Delta}^\cop = \sigma \circ \hat{\Delta} $, where $ \sigma $ denotes the flip map. 
By slight abuse of terminology, we shall refer to both $ \check{G} $ and $ \hat{G} $ as the dual of $ G $. According to Pontrjagin duality, 
the double dual of $ G $ in either of these conventions is canonically isomorphic to $ G $.
\end{rem}

\begin{rem}[cf.\ \cite{KustermansVaes2000}] \label{right reg repr}
The modular conjugations of the left Haar weights $ \phi $ and $ \hat{\phi} $ are denoted by $ J $ and $ \hat{J} $, respectively.
These operators implement the unitary antipodes in the sense that
\[
R(x) = \hat{J} x^* \hat{J}, \qquad \hat{R}(y) = J y^* J
\]
for $ x \in L^\infty(G) $ and $ y \in \mathcal{L}(G) $. Note that $ L^\infty(G)' = JL^\infty(G) J $ and
$ \mathcal{L}(G)' = \hat{J} \mathcal{L}(G) \hat{J} $ for the commutants of $ L^\infty(G) $ and $ \mathcal{L}(G) $. Using $ J $ and $ \hat{J} $ one obtains multiplicative unitaries
\[
V = (\hat{J} \otimes \hat{J})\hat{W}(\hat{J} \otimes \hat{J}), \qquad \hat{V} = (J \otimes J) W (J \otimes J)
\]
which satisfy $ V \in \mathcal{L}(G)' \otimes L^\infty(G) $ and $ \hat{V} \in L^\infty(G)' \otimes \mathcal{L}(G) $, respectively. 
We have 
\[
\Delta(f) = V(f \otimes \eins) V^*, \qquad \hat{\Delta}(x) = \hat{V}(x \otimes \eins) \hat{V}^*
\]
for $ f \in L^\infty(G) $ and $ x \in \mathcal{L}(G) $. We also record the formula 
\[
(\hat{J} \otimes J)W(\hat{J} \otimes J) = W^*,
\]
which is equivalent to saying $ (R \otimes \hat{R})(W) = W $.
\end{rem} 

We will mainly work with the \cstar-algebras associated to the locally compact quantum group $ G $. The
algebra $ [(\id \otimes \IL(L^2(G))_*)(W)] $ is a strongly dense \cstar-subalgebra of $ L^\infty(G) $
which we denote by $ C^\red_0(G) $. Dually, the algebra
$ [(\IL(L^2(G))_* \otimes \id)(W)] $ is a strongly dense \cstar-subalgebra of $ \mathcal{L}(G) $
which we denote by $ C^*_\red(G) $.
These algebras are the reduced algebra of continuous functions vanishing at infinity
on $ G $ and the reduced group \cstar-algebra of $ G $, respectively. One has
$ W \in M(C^\red_0(G) \otimes C^*_\red(G)) $. 

Restriction of the comultiplications on $ L^\infty(G) $ and $ \mathcal{L}(G) $
turns $ C^\red_0(G) $ and $ C^*_\red(G) $ into Hopf-\cstar-algebras in the following sense.
\begin{defi} \label{defhopfcstar}
A Hopf \cstar-algebra is a \cstar-algebra $ H $ together with an injective nondegenerate $ * $-homomorphism
$ \Delta: H \rightarrow M(H \otimes H) $ such that the diagram
\[
\xymatrix{
H \ar@{->}[r]^{\Delta} \ar@{->}[d]^{\Delta} & M(H \otimes H) \ar@{->}[d]^{\id \otimes \Delta} \\
M(H \otimes H) \ar@{->}[r]^{\!\!\!\!\!\!\!\!\! \Delta \otimes \id} & M(H \otimes H \otimes H)
     }
\]
is commutative and $ [\Delta(H)(\eins \otimes H)] = H \otimes H = [(H \otimes \eins)\Delta(H)] $.
\end{defi}

\begin{defi}
A compact quantum group is a locally compact quantum group $ G $ such that $ C^\red_0(G) $ is unital. Similarly, a 
discrete quantum group is a locally compact quantum group $ G $ such that $ C^*_\red(G) $ is unital. 
\end{defi}
We will write $ C^\red(G) $ instead of $ C^\red_0(G) $ if $ G $ is a compact quantum group. 
A finite quantum group is a compact quantum group $ G $ such that $ C^\red(G) $ is finite dimensional. 
This is the case if and only if $ C^*_\red(G) $ is finite dimensional.

\begin{defi}
A unitary corepresentation of a Hopf \cstar-algebra $ H $ on a Hilbert $ B $-module $\CE$ is a unitary 
$ U \in M(H \otimes \IK(\CE)) = \IL(H \otimes \CE) $ such that 
\[
(\Delta \otimes \id)(U) = U_{13} U_{23}. 
\]
A unitary representation of a locally compact quantum group $ G $ is a unitary corepresentation of $ C^\red_0(G) $. 
\end{defi}

\begin{rem} 
This terminology is compatible with the classical case, more precisely, for a classical locally compact group $ G $ 
there is a bijective correspondence between unitary corepresentations of $ C_0(G) $ on Hilbert spaces 
and strongly continuous unitary representations 
of $ G $ in the usual sense, see \cite[Section 5]{MaesvanDaele98} and \cite[5.2.5]{Timmermann08}.
\end{rem} 

\begin{example}
The (left) regular representation of a locally compact quantum group $ G $ is the representation of $ G $ on the Hilbert 
space $ L^2(G) $ given by the 
multiplicative unitary $ W \in M(C^\red_0(G) \otimes \IK(L^2(G))) $, see Remark \ref{Wremark}. In fact, the 
relation $ (\Delta \otimes \id)(W) = W_{13} W_{23} $ is equivalent to the pentagon equation for $ W $. 
\end{example}

\begin{rem} 
Let $ G $ be a locally compact quantum group. The full group \cstar-algebra of $ G $ is a Hopf \cstar-algebra $ C^*_\max(G) $ 
together with a unitary representation $ \mathcal{W} \in M(C^\red_0(G) \otimes C^*_\max(G)) $ satisfying the following universal property: 
for every unitary representation $ U \in M(C_0^\red(G) \otimes \IK(\CE)) $ of $ G $ there exists a unique nondegenerate $ * $-homomorphism 
$ \pi: C^*_\max(G) \rightarrow \IL(\CE) $ such that $ (\id \otimes \pi)(\mathcal{W}) = U $. 

Similarly, one obtains the full \cstar-algebra $ C^\max_0(G) $ of functions on $ G $. 
\end{rem}

\begin{defi}
Let $ G $ be a locally compact quantum group. We say that $ G $ is amenable if the canonical quotient 
map $ C^*_\max(G) \rightarrow C^*_\red(G) $ is an isomorphism. Similarly, $G$ is called coamenable if the dual $\hat{G}$ is amenable. 
\end{defi}

If $G$ is coamenable we will simply write $ C_0(G) $ for $C^\red_0(G)$. By slight abuse of notation, we will also write $ C_0(G) $ if a statement holds for both $ C^\max_0(G) $ and $ C^\red_0(G) $. In particular, if $ G $ is compact and coamenable we will 
simply write $ C(G) $ instead $ C^\red(G) $. 
Similarly, if $G$ is amenable we will write $ C^*(G) $ for $C^*_\red(G)$. We remark that every compact quantum group is amenable, 
and equivalently every discrete quantum group is coamenable.

\begin{rem}[cf.\ \cite{Woronowicz98}, \cite{MaesvanDaele98}] \label{Peter-Weyl}
Let $ G $ be a compact quantum group. In analogy with the theory for compact groups, every unitary representation of $ G $ is completely reducible, and all irreducible representations are finite dimensional. We write $ \Irr(G) $ for the set of equivalence classes of 
irreducible unitary representations of $ G $. Our general second countability assumption amounts to saying that 
the set $ \Irr(G) $ is countable. 

A unitary representation of $ G $ on a finite dimensional Hilbert space $ \CH_\lambda $ is given by a unitary
$ u^\lambda \in C^\red(G) \otimes \IK(\CH_\lambda) $, so it can be viewed as an 
element $ u^\lambda = (u^\lambda_{ij}) \in M_n(C^\red(G)) $ if $ \dim(\CH_\lambda) = n $. 
Moreover, the corepresentation identity translates into the formula 
\[
\Delta(u^\lambda_{ij}) = \sum_{k = 1}^n u^\lambda_{ik} \otimes u^\lambda_{kj}
\]
for the comultiplication of the matrix coefficients $ u^\lambda_{ij} $. 

By Peter-Weyl theory, the linear span of all matrix coefficients $ u^\lambda_{ij} $ for $ \lambda \in \Irr(G) $ forms a 
dense $ * $-subalgebra $ \CO(G) \subset C^\red(G) $. In fact, together with the counit $ \epsilon: \CO(G) \to \mathbb{C} $ given by 
\[
\epsilon(u^\lambda_{ij}) = \delta_{ij} 
\]
and the antipode $ S: \CO(G) \to \CO(G) $ given by 
\[
S(u^\lambda_{ij}) = (u^\lambda_{ij})^* 
\]
the algebra $ \CO(G) $ becomes a Hopf $ * $-algebra. 

We will use the Sweedler notation $ \Delta(f) = f_{(1)} \otimes f_{(2)} $ for the comultiplication of general elements of $ \CO(G) $. 
This is useful for bookkeeping of coproducts, let us emphasize however that this notation is not meant to say that $ \Delta(f) $ is a simple tensor. For higher coproducts one introduces further indices, for instance $ f_{(1)} \otimes f_{(2)} \otimes f_{(3)} $ is an abbreviation 
for $ (\Delta \otimes \id) \Delta(f) = (\id \otimes \Delta) \Delta(f) $.

Again by Peter-Weyl theory, one obtains a vector space basis of $ \CO(G) $ consisting of the matrix coefficients $ u^\lambda_{ij} $ 
where $ \lambda $ ranges over $ \Irr(G) $ and $ 1 \leq i,j \leq \dim(\lambda) $. Here we abbreviate $ \dim(\lambda) = \dim(\CH_\lambda) $. 
If we write $ \CO(G)_\lambda $ for the linear span of the elements $ u^\lambda_{ij} $ for $ 1 \leq i,j \leq \dim(\lambda) $, 
we have a direct sum decomposition 
\[
\CO(G) = \text{alg-}\bigoplus_{\lambda \in \Irr(G)} \CO(G)_\lambda. 
\]
Note that the coproduct of $ \CO(G) $ takes a particularly simple form in this picture; from the above formula for 
$ \Delta(u^\lambda_{ij}) $ we see that it looks like the transpose of matrix multiplication. 

Let $ u^\lambda $ be an irreducible unitary representation of $ G $, and let $ u^\lambda_{ij} $ be the corresponding 
matrix elements in some fixed basis. The contragredient representation $ u^{\lambda^c} $ is given by the 
matrix $ (u^{\lambda^c})_{ij} = S(u^\lambda_{ji}) $ where $ S $ is the antipode 
of $ \CO(G) $. In general $ u^{\lambda^c} $ is not unitary, but as any finite dimensional representation of $ G $, it is 
unitarizable. The representations $ u^\lambda $ and $ u^{\lambda^{cc}} $ are equivalent, and there exists a unique positive invertible 
intertwiner $ F_\lambda: \CH_\lambda \to \CH_{\lambda^{cc}} $ satisfying $ tr(F_\lambda) = tr(F_\lambda^{-1}) $. 
The trace of $ F_\lambda $ is called the quantum dimension of $ \CH_\lambda $ and denoted by $ \dim_q(\lambda) $. 

With this notation, the Schur orthogonality relations are
\[
\phi(u^\lambda_{ij} (u^\eta_{kl})^*) = \delta_{\lambda\eta} \delta_{ik} \, \frac{1}{\dim_q(\lambda)}\, (F_\lambda)_{lj}
\]
where $ \lambda, \eta \in \Irr(G) $ and $ \phi: C^\red(G) \rightarrow \mathbb{C} $ is the Haar state of $ G $. 
In the sequel we shall fix bases such that $ F_\lambda $ is a diagonal operator for all $ \lambda \in \Irr(G) $. 
\end{rem}

\begin{rem} \label{Peter-Weyl2}
Let again $ G $ be a compact quantum group. While the comultiplication for the \cstar-algebra $ C^\red(G) $ looks particularly simple in 
terms of matrix coefficients, dually the multiplication in the \cstar-algebra $ C^*(G) $ is easy to describe. 
More precisely, we have
\[
C^*(G) \cong \bigoplus_{\lambda \in \Irr(G)} \IK(\CH_\lambda), 
\]
where the right hand side denotes the $ c_0 $-direct sum of the matrix algebras $ \IK(\CH_\lambda) $.  
If $ u^\lambda_{ij} $ are the matrix coefficients in $ \CO(G) $, then the dual basis vectors $ \omega^\lambda_{ij} $, 
that is, the linear functionals on $ \CO(G) $ given by 
\[
\omega^\lambda_{ij}(u^\eta_{kl}) = \delta_{\lambda \eta} \delta_{ik} \delta_{jl} 
\]
form naturally a vector space basis of matrix units for the algebraic direct sum 
\[
\CD(G) = \text{alg-}\bigoplus_{\lambda \in \Irr(G)} \IK(\CH_\lambda)
\]
inside $ C^*_\red(G) $.

Let us note that $\CD(G)$ should not be confused with a quantum double, but there will be no conflicting notation in this regard appearing in this paper.

Let us also note that according to the Schur orthogonality relations, see Remark \ref{Peter-Weyl}, 
the functionals $ \omega^\lambda_{ij} $ extend continuously to bounded linear functionals on $ C^\red(G) $. 
\end{rem}

\subsection{Actions, crossed products and Takesaki-Takai duality}

Let us now consider actions on \cstar-algebras.

\begin{defi} \label{defcoaction}
A (left) coaction of a Hopf \cstar-algebra $ H $ on a \cstar-algebra $ A $ is an injective nondegenerate $ * $-homomorphism
$ \alpha: A \rightarrow M(H \otimes A) $ such that the diagram
\[
\xymatrix{
A \ar@{->}[r]^{\alpha} \ar@{->}[d]^\alpha & M(H \otimes A) \ar@{->}[d]^{\Delta \otimes \id} \\
M(H \otimes A) \ar@{->}[r]^{\!\!\!\!\!\!\!\!\!\! \id \otimes \alpha} & M(H \otimes H \otimes A)
     }
\]
is commutative and $ [\alpha(A)(H \otimes \eins)] = H \otimes A $.

If $ (A, \alpha) $ and $ (B, \beta) $ are \cstar-algebras equipped with coactions of $ H $, then a $ * $-homomorphism
$ \phi: A \rightarrow M(B) $ is called $ H $-colinear if $ \beta \phi = (\id \otimes \phi)\alpha $.

An action of a locally compact quantum group $ G $ on a \cstar-algebra $ A $ is a 
coaction $ \alpha: A \rightarrow M(C^\red_0(G) \otimes A) $ of $ C^\red_0(G) $ on $ A $. We will also say
that $ (A, \alpha) $ is a $ G $-\cstar-algebra in this case. A $*$-homomorphism $\phi:A \to M(B)$ between $G$-\cstar-algebras is 
called $G$-equivariant if it is $C_0^\red(G)$-colinear.
\end{defi}

If $ H $ is a Hopf \cstar-algebra and $\alpha: A \rightarrow M(H \otimes A)$ a coaction, then 
the density condition $[\alpha(A)(H \otimes \eins)] = H \otimes A $ implies in particular that the image of  
$\alpha$ is contained in the $H$-relative multiplier algebra of $H \otimes A$, defined by 
\[
M_H(H \otimes A) = \set{m \in M(H \otimes A) \mid m(H\otimes \eins), (H \otimes \eins) m \subset H \otimes A}.
\]

In contrast to the situation for ordinary multiplier algebras, the tensor product map $\id \otimes \phi: H \otimes A \to H \otimes B $
of a (possibly degenerate) $*$-homomorphism $\phi:A \to B$ admits a unique extension 
$ M_H(H \otimes A) \to M_H(H \otimes B)$ to the relative multiplier algebra, which will again be denoted by $ \id \otimes \phi $. 
If $\phi$ is injective, then this also holds for $\id \otimes \phi: M_H(H \otimes A) \to M_H(H \otimes B)$. 
We refer to \cite[Appendix A]{EKQR2006} for further details on relative multiplier algebras.

\begin{rem} \label{induct limit coactions}
For the construction of examples of Rokhlin actions we shall consider inductive limit actions of Hopf \cstar-algebras. 
Assume $ H $ is a Hopf \cstar-algebra and that $ A_1 \rightarrow A_2 \rightarrow \cdots $ is an inductive system of $ H $-\cstar-algebras with coactions $ \alpha_j: A_j \rightarrow M_H(H \otimes A_j) $ and injective equivariant connecting maps. Then the direct limit $ A = \varinjlim A_j $ becomes an $ H $-\cstar-algebra in a canonical way. 
Firstly, we have $ \varinjlim M_H(H \otimes A_j) \subset M_H(H \otimes A) $ naturally since 
our system $ (A_j)_{j \in \mathbb{N}} $ 
has injective connecting maps. The maps $ \alpha_j $ define a compatible family of $ * $-homomorphisms 
$ \alpha_j: A_j \rightarrow M_H(H \otimes A_j) \rightarrow M_H(H \otimes A) $, and induce 
a $ * $-homomorphism $ \alpha: A \rightarrow M_H(H \otimes A) $. Coassociativity of $ \alpha $ and the density 
condition $ [(H \otimes \eins)\alpha(A)] = H \otimes A $ follow from the corresponding properties of the coactions $ \alpha_j $. Injectivity of 
the map $ \alpha: A \rightarrow M_H(H \otimes A) $ follows from injectivity of the maps $ \alpha_j $ and of the connecting maps in the 
system $ H \otimes A_j $.
\end{rem}

\begin{rem} [cf.\ \cite{Podles95}] \label{spectralsub} 
Let $ G $ be a compact quantum group. Further below we will use the fact that any $ G $-\cstar-algebra $A$ 
admits a spectral decomposition. In order to discuss this we review some further definitions 
and results. 

Let $ G $ be a compact quantum group and let $ (A, \alpha) $ be a $ G $-\cstar-algebra.  
Since $ G $ is compact, the coaction is an injective $ * $-homomorphism $ \alpha: A \rightarrow C^\red(G) \otimes A $ 
satisfying the coassociativity identity $ (\Delta \otimes \id) \circ \alpha = (\id \otimes \alpha) \circ \alpha $ 
and the density condition $ [(C^\red(G) \otimes \eins)\alpha(A)] = C^\red(G) \otimes A $. For $ \lambda \in \Irr(G) $ we let 
\[
A_\lambda = \{a \in A \mid \alpha(a) \in \CO(G)_\lambda \odot A \}
\]
be the $ \lambda $-spectral subspace of $ A $. Here we recall that $ \CO(G)_\lambda \subset \CO(G) $ denotes the 
span of the matrix coefficients for $ \CH_\lambda $, see Remark \ref{Peter-Weyl}. 

The subspace $ A_\lambda $ is closed in $ A $, and there is a projection operator $ p_\lambda: A \rightarrow A_\lambda $ defined by 
\[
p_\lambda(a) = (\theta_\lambda \otimes \id) \alpha(a)
\]
where 
\[ 
\theta_\lambda(x) = \dim_q(\lambda) \sum_{j = 1}^{\dim(\lambda)} (F_\lambda)^{-1}_{jj} \phi(x (u_{jj}^\pi)^*). 
\]
By definition, the spectral subalgebra $ \CS(A) \subset A $ is the $ * $-subalgebra 
\[
\CS(A) = \CS_G(A) = \text{alg-}\bigoplus_{\lambda \in \Irr(G)} A_\lambda. 
\]
>From the Schur orthogonality relations and $ [(C^\red(G) \otimes \eins)\alpha(A)] = C^\red(G) \otimes A $ 
it is easy to check that $ \CS(A) $ is dense in $ A $. 

For $ \omega \in \CD(G) $ and $ a \in A $ let us define 
\[
a \cdot \omega = (\omega \otimes \id)\circ \alpha(a).  
\]
Then the Schur orthogonality relations imply that $ A \cdot \CD(G) $, the linear span of all elements 
$ a \cdot \omega $ as above, is equal to $ \CS(A) $. Moreover, from the coaction property of $ \alpha $ it follows that 
$ A $ becomes a right $ \CD(G) $-module in this way. 
\end{rem}

\begin{defi} \label{crossed product}
Let $ G $ be a locally compact quantum group and let $ A $ be a $ G $-\cstar algebra with 
coaction $ \alpha: A \rightarrow M(C^\red_0(G) \otimes A) $. 
The reduced crossed product $ G \ltimes_{\alpha, \red} A $ is the \cstar-algebra 
\[
G \ltimes_{\alpha, \red} A = [(C_0^\red(\check{G}) \otimes \eins)\alpha(A)] \subset M(\IK(L^2(G)) \otimes A).
\]
\end{defi}

Recall here that $ C_0^\red(\check{G}) = C^*_\red(G) $ as a \cstar-algebra, but equipped with the opposite comultiplication. 

The reduced crossed product is equipped with a canonical dual action of $ \check{G} $, which turns it into a $ \check{G} $-\cstar-algebra. More precisely, the dual action is given by comultiplication on the copy of $ C_0^\red(\check{G}) $ and the trivial action on the copy 
of $ A $ inside $ M(G \ltimes_{\alpha, \red} A) $. By our second countability assumption, the crossed product $ G \ltimes_{\alpha, \red} A $ 
is separable provided $ A $ is separable.

For the purpose of reduced duality, we have to restrict ourselves to regular locally compact quantum groups \cite{BaajSkandalis93}. 
All compact and discrete quantum groups are regular, so this is not an obstacle for the constructions we are interested in further below. 

\begin{rem} \label{conjug action compacts}
If $ G $ is a regular locally compact quantum group then the regular representation of $ L^2(G) $ induces an action of $ G $ on the 
algebra $ \IK(L^2(G)) $ of compact operators by conjugation. More generally, if $ A $ is any $ G $-\cstar-algebra we can turn the tensor product $ \IK(L^2(G)) \otimes A $ into a $ G $-\cstar-algebra by equipping $ \IK(L^2(G)) \otimes A \cong \IK(L^2(G) \otimes A) $ with the conjugation action arising from the tensor product representation of $ G $ on the Hilbert $ A $-module $ \CE = L^2(G) \otimes A $. Explicitly, following the notation in \cite{NestVoigt2010}, we consider 
\[
\lambda(\xi \otimes a) = X^*_{12} \Sigma_{12} (\id \otimes \alpha)(\xi \otimes a) 
\] 
where $ X = \Sigma V \Sigma $ and $V$ is as in Remark \ref{right reg repr}. This determines a 
coaction $ \lambda: \CE \rightarrow M(C^\red_0(G) \otimes \CE) $, which in turn corresponds to a unitary corepresentation 
\[
V_\lambda^* \in \IL(C^\red_0(G) \otimes \CE) \cong M(C^\red_0(G) \otimes \IK(\CE)).
\]
The conjugation action $ \ad_\lambda: \IK(\CE) \rightarrow M(C^\red_0(G) \otimes \IK(\CE)) \cong \IL(C^\red_0(G) \otimes \CE) $ is then defined by 
\[
\ad_\lambda(T) = V_\lambda(\eins \otimes T) V_\lambda^*,
\]
and $ \IK(\CE) $ becomes a $ G $-\cstar-algebra in this way. Under the isomorphism $ \IK(L^2(G)) \otimes A \cong \IK(L^2(G) \otimes A) $, this $G$-action is given by
\[
\alpha_\IK(T \otimes a) = \ad_\lambda(T \otimes a) 
= X^*_{12}(\eins \otimes T \otimes \eins)\alpha(a)_{13} X_{12}, \quad T\in \IK(L^2(G)), a \in A.
\]
If $ G $ is a classical group then the resulting action is nothing but the diagonal action on the tensor product $ \IK(L^2(G)) \otimes A $. For further details we refer to \cite{BaajSkandalis89}, \cite{NestVoigt2010}.
\end{rem}

Let us now state the Takesaki-Takai duality theorem for regular locally compact quantum groups \cite{BaajSkandalis93}, 
see \cite[chapter 9]{Timmermann08} for a detailed exposition. 

\begin{theorem} \label{TTduality}
Let $ G $ be a regular locally compact quantum group and let $ A $ be a $ G $-\cstar-algebra. Then there is a natural isomorphism
\[
(\check{G} \ltimes_{\check{\alpha}, \red} G \ltimes_{\alpha, \red} A, \check{\check{\alpha}}) \cong (\IK(L^2(G)) \otimes A, \alpha_\IK)
\]
of $ G $-\cstar-algebras.
\end{theorem}

Let us give a brief sketch of the proof of Theorem \ref{TTduality} for the sake of convenience. 
Recall from Remark \ref{right reg repr} that $ J $ is the modular conjugation of the left Haar weight of $ L^\infty(G) $, 
and similarly $ \hat{J} $ the modular conjugation of the left Haar weight of $ \mathcal{L}(G) $. We shall write $ U = J \hat{J} $. 
 
By definition, we have 
\[
\check{G} \ltimes_{\check{\alpha}, \red} G \ltimes_{\alpha, \red} A 
= [(U C_0^\red(G) U \otimes \eins \otimes \eins)(\hat{\Delta}^\cop(C^*_\red(G)) \otimes \eins)\alpha(A)_{23}]. 
\]
Conjugating by $ W_{12}^* $ and using that $ C_0^\red(G) $ and $ J C_0^\red(G) J $ commute, this is isomorphic to 
\[ 
[(U C_0^\red(G) U \otimes \eins \otimes \eins)(C^*_\red(G)) \otimes \eins \otimes \eins)(\id \otimes \alpha)\alpha(A)]
\cong [(U C_0^\red(G) U C^*_\red(G)) \otimes \eins) \alpha(A)] 
\] 
since $ \alpha $ is injective. Moreover, $ [U C_0^\red(G) U C^*_\red(G)] $ identifies with $ \IK(L^2(G)) $. 
Hence the right hand side is isomorphic to $ \IK(L^2(G)) \otimes A $, taking into account that 
\[
[(\IK(L^2(G)) \otimes \eins) \alpha(A)] = [(\IK(L^2(G)) C_0^\red(G) \otimes \eins)\alpha(A)] = [\IK(L^2(G)) C_0^\red(G) \otimes A]  
\] 
by the density condition $ [(C_0^\red(G) \otimes \eins) \alpha(A)] = C_0^\red(G) \otimes A $.
 
Let us also identify the bidual coaction. By construction, the bidual coaction $ \check{\check{\alpha}} $ 
maps $ UfU \otimes \eins \in \IK(L^2(G)) \otimes A $ 
for $ f \in C^\red_0(G) $ to $ (\eins \otimes U)\Delta(f) (\eins \otimes U) $ and leaves $ C^*_\red(G) \otimes \eins $ and $ \alpha(A) $ 
fixed. Using the relations from Remark \ref{right reg repr} one can show 
\[ 
\check{\check{\alpha}}(T \otimes a) = X_{12}^* (\eins \otimes T \otimes \eins) \alpha(a)_{13} X_{12}, 
\] 
where $ X = (\eins \otimes U) W (\eins \otimes U) $.

At some points we will also need the full crossed product $ G \ltimes_{\alpha, \max} A $ of a $ G $-\cstar-algebra $ (A, \alpha) $; 
we refer to \cite{NestVoigt2010} for a review of its definition in terms of its universal property for covariant representations.

\subsection{Braided tensor products}

Finally, let us discuss Yetter-Drinfeld actions and braided tensor products. 
We refer to \cite{NestVoigt2010} for more information.

\begin{defi} \label{Yetter-Drinfeld}
Let $ G $ be a locally compact quantum group. A $ G$-$ \yd $-\cstar-algebra is a \cstar-algebra $ A $ together with a pair of 
actions $ \alpha: A \rightarrow M(C^\red_0(G) \otimes A) $ and $ \gamma: A \rightarrow M(C^*_\red(G) \otimes A) $ satisfying the 
Yetter-Drinfeld compatibility condition
\[
(\sigma \otimes \id) \circ (\id \otimes \alpha) \circ \gamma = (\ad(W) \otimes \id) \circ (\id \otimes \gamma) \circ \alpha. 
\]
Here $ W \in M(C^\red_0(G) \otimes C^*_\red(G) $ is the multiplicative unitary.  
\end{defi}

The braided tensor product, which we review next, generalizes at the same time the minimal tensor product 
of \cstar-algebras and reduced crossed products. Roughly speaking, it allows one to construct a new \cstar-algebra 
out of two constituent \cstar-algebras, together with some prescribed commutation relations between the two factors. 
\begin{defi} 
Let $ G $ be a locally compact quantum group. Given a $ G $-$ \yd $-\cstar-algebra $ A $ and a $ G $-\cstar-algebra $ (B,\beta) $, 
one defines the braided tensor product $ A \boxtimes_G B = A \boxtimes B $ by 
\[
A \boxtimes B = [\gamma(A)_{12} \beta(B)_{13}] \subset M(\IK(L^2(G)) \otimes A \otimes B). 
\]
\end{defi}

The braided tensor product $ A \boxtimes B $ becomes a $ G $-\cstar-algebra with coaction 
$ \alpha \boxtimes \beta: A \boxtimes B \rightarrow M(C^\red_0(G) \otimes A \boxtimes B) $ in such a way that the 
canonical embeddings $ \iota_A: A \rightarrow M(A \boxtimes B) $ and $ \iota_B: B \rightarrow M(A \boxtimes B) $ are $ G $-equivariant.
We write $ a \boxtimes \eins $ and $ \eins \boxtimes b $ for the images of elements $ a \in A $ and $ b \in B $ in $ M(A \boxtimes B) $, 
respectively. 

\begin{example} \label{exyd}
A basic example of a $G$-$ \yd $-action is given by the \cstar-algebra $ A = C^\red_0(G) $ for a regular locally compact quantum group $ G $, 
equipped with $ \alpha = \Delta $ and the adjoint action $ \gamma(f) = \hat{W}^*(\eins \otimes f) \hat{W} $. 
\end{example}
 
We will mainly be interested in the special case of the braided tensor product
construction where the first factor is equal to $ C^\red_0(G) $ with the Yetter-Drinfeld structure from Example \ref{exyd}.

\begin{lemma} \label{trivialization}
Let $ G $ be a regular locally compact quantum quantum group. For any $ G $-\cstar-algebra $ A $, there 
exists a $G$-equivariant $ * $-isomorphism 
\[ 
T_\alpha: (C_0^\red(G) \boxtimes A, \Delta \boxtimes \alpha) \rightarrow (C_0^\red(G) \otimes A, \Delta \otimes \id)
\]
such that $ T_\alpha(1 \boxtimes a) = \alpha(a) $ for all $ a \in A \subset M(C_0^\red(G) \boxtimes A) $ 
and $ T_\alpha(f \boxtimes \eins) = f \otimes \eins $ for $ f \in C^\red_0(G) \subset M(C_0^\red(G) \boxtimes A) $. 
\end{lemma}
\begin{proof} The map $ T_\alpha $ is obtained from the identifications 
\[
\begin{array}{lcl}
C^\red_0(G) \boxtimes A &=& [\hat{W}_{12}^* (\eins \otimes C^\red_0(G) \otimes \eins) \hat{W}_{12} \alpha(A)_{13}] \\ 
&\cong& [(\eins \otimes C^\red_0(G) \otimes \eins) \hat{W}_{12} \alpha(A)_{13} \hat{W}_{12}^*] \\ 
&\cong& [(C^\red_0(G) \otimes \eins \otimes \eins) W^*_{12} \alpha(A)_{23} W_{12}] \\ 
&=& [(C^\red_0(G) \otimes \eins \otimes \eins) (\id \otimes \alpha)\alpha(A)] \\ 
&\cong& [(C^\red_0(G) \otimes \eins) \alpha(A)]\\
&=& C^\red_0(G) \otimes A,
\end{array}
\]
where we use $ \hat{W} = \Sigma W^* \Sigma $ in the third step and the fact that $ \alpha $ is injective in the penultimate step. 

It is straightforward to check that $ T_\alpha(a) = \alpha(a) $ for $ a \in A \subset M(C_0^\red(G) \boxtimes A) $ 
and $ T_\alpha(f) = f \otimes \eins $ for $ f \in C_0^\red(G) \subset M(C_0^\red(G) \boxtimes A) $.

In particular, the canonical action of $ G $ on $ C^\red_0(G) \boxtimes A $ corresponds to the translation 
action on the first tensor factor in $ C^\red_0(G) \otimes A $ under this isomorphism.
\end{proof}


\section{Induced actions on sequence algebras}

The theory of Rokhlin actions for compact quantum groups relies on the possibility of obtaining induced actions on the level of 
sequence algebras. In this section we shall recall a few facts on sequence algebras, 
and then discuss the construction of induced actions, separately for the case of discrete and compact quantum groups.

\subsection{Sequence algebras} 
Let us first recall some notions related to sequence algebras, see \cite{BarlakSzabo15} and \cite{Kirchberg04}.
If $A$ is a \cstar-algebra we write 
\[
\ell^\infty(A) = \{(a_n)_n \mid a_n \in A \text{ and } \sup_{n \in \IN} \|a_n \| < \infty \}
\]
for the \cstar-algebra of bounded sequences with coefficients in $A$.
Moreover denote by 
\[
c_0(A) = \set{(a_n)_n\in\ell^\infty(A) \mid \lim_{n\to\infty} \|a_n\|=0 } \subset \ell^\infty(A)
\]
the closed two-sided ideal of sequences converging to zero. 
\begin{defi}
Let $A$ be a \cstar-algebra. The sequence algebra $A_\infty$ of $A$ is 
\[
A_\infty = \ell^\infty(A)/c_0(A).  
\]
\end{defi}
Given a bounded sequence $(a_n)_{n \in \mathbb{N}} \in \ell^\infty(A)$, the norm of the corresponding element in $A_\infty$ is given by
\[
\| [(a_n)_{n \in \mathbb{N}}] \| = \limsup_{n\to \infty} \|a_n \|.
\]
Note moreover that $A$ embeds canonically into $A_\infty$ as (representatives of) constant sequences. We will frequently use this identification of $A$ inside $A_\infty$ in the sequel.

\begin{nota}
We denote by
\[
\ann_\infty(A) = \set{ x\in A_\infty \mid xa=ax=0 \text{ for all } a\in A}
\]
the two-sided annihilator of $A$ inside $A_\infty$. Moreover, we write
\[
D_{\infty,A}=[A\cdot A_\infty\cdot A]\subset A_\infty
\]
for the hereditary subalgebra of $A_\infty$ generated by $A$. Note that the embedding $A \into D_{\infty,A}$ is clearly nondegenerate. Finally, consider also the normalizer of $D_{\infty,A}$ inside $A_\infty$,
\[
\CN(D_{\infty,A},A_\infty) = \set{ x\in A _\infty~|~ xD_{\infty,A}+D_{\infty,A}x \subset D_{\infty,A}}.
\]
We remark that $\ann_\infty(A)$ sits inside $\CN(D_{\infty,A},A_\infty)$ as a closed two-sided ideal.
\end{nota}

The multiplier algebra of $D_{\infty,A}$ admits the following alternative description.

\begin{prop}[cf.~{\cite[Proposition 1.5(1)]{BarlakSzabo15}} and {\cite[1.9(4)]{Kirchberg04}}] 
\label{surjection normalizer multiplier}
Let $A$ be a $\sigma$-unital \cstar-algebra. The canonical $*$-homomorphism $ \CN(D_{\infty,A},A_\infty) \to M(D_{\infty,A})$ given by the universal property of the multiplier algebra is surjective, and its kernel coincides with $\ann_\infty(A)$.
\end{prop}

Let us also note that the construction of $ D_{\infty, A} $ is compatible with tensoring by the compacts. 

\begin{lemma}[cf.\ {\cite[1.6]{BarlakSzabo15}}] \label{lemmaDomegastable}
Let $ A $ be a \cstar-algebra and let $ \IK(\CH) $ be the algebra of compact operators on a separable Hilbert space $ \CH $. The canonical embedding $ \IK(\CH) \otimes A_\infty \to (\IK(\CH) \otimes A)_\infty $ induces an isomorphism $ \IK(\CH) \otimes D_{\infty, A} 
\cong D_{\infty, \IK(\CH) \otimes A} $. 
\end{lemma}

Given a \cstar-algebra equipped with an action of a quantum group $ G $, we shall now discuss how to obtain induced actions 
on the sequence algebras introduced above. 

\subsection{Induced actions - discrete case}

In the case of discrete quantum groups the situation is relatively simple. In fact, if $ G $ is a discrete 
quantum group then the \cstar-algebra $ C_0(G) $ of functions on $ G $ is a \cstar-direct sum of matrix algebras. 
Explicitly, it is of the form 
\[
C_0(G) \cong \bigoplus_{\lambda \in \Lambda} \IK(\CH_\lambda)
\]
where $ \Lambda = \Irr(\check{G}) $ is the set of equivalence classes of irreducible representations of the dual compact quantum group, 
see Remark \ref{Peter-Weyl}.

If $ (A, \alpha) $ is a $ G $-\cstar-algebra, then one of the defining conditions for the 
coaction $ \alpha: A \rightarrow M(C_0(G) \otimes A) $ is that it factorizes over the $ C_0(G) $-relative multiplier 
algebra $ M_{C_0(G)}(C_0(G) \otimes A) $.  
With the notation as above, we have 
\[
M_{C_0(G)}(C_0(G) \otimes A) \cong \prod_{\lambda \in \Lambda} \IK(\CH_\lambda) \otimes A, 
\]
that is, we can identify the relative multiplier algebra with the $ \ell^\infty $-product of the algebras $ \IK(\CH_\lambda) \otimes A $. 
In other words, we have 
\[ 
\alpha: A \to \prod_{\lambda \in \Lambda} \IK(\CH_\lambda) \otimes A \subset M(C_0(G) \otimes A). 
\] 
It follows that applying $ \alpha $ componentwise induces 
a $ * $-homomorphism $ \alpha^\infty: \ell^\infty(A) \to M(C_0(G) \otimes \ell^\infty(A)) $ 
by considering the composition 
\[
\begin{array}{lcl} 
\ell^\infty(A) & \to & \prod_{n \in \IN} \prod_{\lambda \in \Lambda} \IK(\CH_\lambda) \otimes A \\
& \cong & \prod_{\lambda \in \Lambda} \prod_{n \in \IN} \IK(\CH_\lambda) \otimes A \\ 
& \cong & \prod_{\lambda \in \Lambda} \IK(\CH_\lambda) \otimes \prod_{n \in \IN} A \subset M(C_0(G) \otimes \ell^\infty(A)), 
\end{array}
\]
here we use that $ \IK(\CH_\lambda) $ is finite dimensional for all $ \lambda \in \Lambda $ in the second isomorphism. 

\begin{lemma} 
Let $ G $ be a discrete quantum group and let $ (A,\alpha) $ be a $ G $-\cstar-algebra. 
Then the map $ \alpha^\infty: \ell^\infty(A) \to M(C_0(G) \otimes \ell^\infty(A)) $ constructed above 
turns $ \ell^\infty(A) $ into a $ G $-\cstar-algebra. 
\end{lemma} 
\begin{proof} Injectivity and coassociativity of $ \alpha^\infty $ follow immediately from the corresponding properties of $ \alpha $. 
For the density condition 
\[
[(C_0(G) \otimes \eins)\alpha^\infty(\ell^\infty(A))] = C_0(G) \otimes \ell^\infty(A)
\]
notice that it suffices to verify
\[
[(p C_0(G) \otimes \eins) \alpha^\infty(\ell^\infty(A))] = pC_0(G) \otimes \ell^\infty(A) 
\] 
for all finite rank (central) projections $ p \in C_0(G) $. This in turn follows from the density condition for $ \alpha $, 
combined with fact that tensoring with finite dimensional 
algebras commutes with taking direct products.
\end{proof}

The map $ \alpha^\infty $ constructed above induces an injective $*$-homomorphism $ \alpha_\infty: A_\infty \to M(C_0(G) \otimes A_\infty) $ 
that fits into the following commutative diagram 

\[
\xymatrix{
\ell^\infty (A) \ar@{->}[r] \ar@{->}[d]^{\alpha^\infty} 
& A_\infty \ar@{->}[d]^{\alpha_\infty} \\
M_{C_0(G)}(C_0(G) \otimes \ell^\infty(A)) \ar@{->}[r] 
& M_{C_0(G)}(C_0(G) \otimes A_\infty) 
     }
\]

Indeed, it suffices to observe that $ \alpha^\infty $ maps $c_0(A) $ 
into $ \prod_{\lambda \in \Irr(G)} \IK(\CH_\lambda) \otimes c_0(A) $. 
Coassociativity and the density 
conditions for $ \alpha_\infty $ are inherited from $ \alpha^\infty$. We therefore obtain the following result.

\begin{lemma} 
Let $ G $ be a discrete quantum group and let $ (A,\alpha) $ be a $ G $-\cstar-algebra. 
Then the map $ \alpha_\infty: A_\infty \to M(C_0(G) \otimes A_\infty) $ turns $ A_\infty $ into a $ G $-\cstar-algebra. 
\end{lemma}

\subsection{Induced actions - compact case}

Whereas for discrete quantum groups the extension of actions to sequence algebras always yields genuine actions, 
the situation for compact quantum groups 
is more subtle. Already classically, a strongly continuous action of a compact group $ G $ on a \cstar-algebra $ A $ induces 
an action on $ \ell^\infty(A) $ and $ A_\infty $, but these induced actions typically fail to be strongly continuous, 
compare \cite{BarlakSzabo15}. 

We shall address the corresponding problems in the quantum case by using an ad-hoc 
notion of equivariant $*$-homomorphisms into sequence algebras. Our discussion also requires the technical assumption of coexactness.

\begin{defi}
A locally compact quantum group $G$ is exact if the functor of taking reduced crossed products by $ G $ 
is exact. We say that $ G $ is coexact if the dual $ \hat{G} $ is exact. 
\end{defi}

It is well-known that a discrete quantum group $G$ is exact if and only if $ C^*_\red(G) $ is an exact \cstar-algebra,
see \cite[1.28]{VaesVergnioux07}. In other words, a compact quantum group $ G $ is coexact if and only if $ C^\red(G) $ 
is an exact \cstar-algebra. 

\begin{rem} \label{exactnessrem}
Let $A$ and $B$ be \cstar-algebras. If $B$ is exact, then there exists a canonical injective $ * $-homomorphism $B \otimes A_\infty \into (B \otimes A)_\infty $ coming from the following commutative diagram with exact rows
\[
\xymatrix{
0 \ar@{->}[r] & B \otimes c_0(A) \ar@{->}[r] \ar@{->}[d]^\cong & B \otimes \ell^\infty (A) \ar@{->}[r] \ar@{^(->}[d] & B \otimes A_\infty \ar@{->}[r] \ar@{^(->}[d] & 0 \\
0 \ar@{->}[r] & c_0(B \otimes A) \ar@{->}[r] & \ell^\infty(B \otimes A) \ar@{->}[r] 
& (B \otimes A)_\infty \ar@{->}[r] & 0
     }
\]

For similar reasons, we have a chain of natural inclusions $B \otimes M(D_{\infty,A}) \subset M(B \otimes D_{\infty,A}) \subset M(D_{\infty,B \otimes A})$ provided $B$ is exact.
\end{rem}

Now let $ G $ be a compact quantum group and let $ (A, \alpha) $ be a $G$-\cstar-algebra. 
The $ * $-homomorphisms $\alpha: A\to C^\red(G)\otimes A$ induces a $*$-homomorphism
\[
\alpha^\infty: \ell^\infty(A) \to \ell^\infty(C^\red(G) \otimes A), 
\]
obtained by applying $ \alpha $ componentwise, and also a $ * $-homomorphism
\[
\alpha_\infty: A_\infty \to (C^\red(G) \otimes A)_\infty. 
\]
The maps $ \alpha^\infty $ and $ \alpha_\infty $, despite not being coactions in the sense of Definition \ref{defcoaction} in general, turn out to be good enough to obtain a tractable notion of equivariance and suitable crossed products, at least when $G$ is coexact. The reason for this is Remark \ref{exactnessrem}, which is used in the definition below.

\begin{defi} \label{equiv sequence alg}
Let $ G $ be a coexact compact quantum group and let $(A,\alpha)$ and $ (B,\beta) $ be $ G $-\cstar-algebras. 
A $*$-homomorphism $\phi: A \to B_\infty$ is said to be $G$-equivariant if $\beta_\infty \circ \phi = (\id \otimes \phi) \circ \alpha$, 
where both sides are viewed as maps from $ A $ into $ (C^\red(G) \otimes B)_\infty $. 
If $\phi$ is $G$-equivariant, then we also write $\phi:(A,\alpha) \to (B_\infty,\beta_\infty)$.
\end{defi}

Note in particular that if $ \phi: A \to B_\infty $ is a $G$-equivariant $*$-homomorphism, then we 
automatically have $\beta_\infty \circ \phi(A) \subset C^\red(G) \otimes B_\infty$. 

\begin{rem} \label{rem equiv class group}
As indicated above, if $G$ is a compact group and $\alpha:G \curvearrowright A$ is a strongly continuous action on a \cstar-algebra, 
there always exists a (not necessarily strongly continuous) induced action of $G$ on $A_\infty$. If $(B, \beta)$ is another $G$-algebra, 
then it is easy to see that any $*$-homomorphism $\phi:A \to B_\infty$ that is $G$-equivariant in the sense \ref{equiv sequence alg} 
if and only if it is $G$-equivariant in the usual sense. 

Indeed, the equality $ (\id \otimes \phi) \circ \alpha = \beta_\infty \circ \phi $ clearly implies equivariance 
the usual sense. For the converse implication, one notes that if $ \phi $ is equivariant in the usual sense, it 
maps $ A $ automatically into the continuous part of the action 
on $ B_\infty $. Therefore, for $ a \in A $ the equality of $ (\id \otimes \phi) \circ \alpha(a) = \beta_\infty \circ \phi(a) $ 
in $ C(G) \otimes B_\infty = C(G, B_\infty) $ can be checked by evaluating both sides at the points of $ G $. 
\end{rem}

In the quantum setting, we need a substitute of the continuous part of an action in order to define crossed products. 
We shall rely on the structure of compact quantum groups to obtain a construction suitable for the situation at hand.

Recall from Remark \ref{Peter-Weyl} that the dense Hopf $ * $-algebra $ \CO(G) \subset C^\red(G) $ has a 
linear basis of elements of the form $ u^\lambda_{ij} $ where $ \lambda \in \Irr(G) $ and $ 1 \leq i,j, \leq \dim(\lambda) $. 
As explained in Remark \ref{Peter-Weyl2}, the linear functionals $ \omega_{ij}^\lambda \in C^\red(G)^*$ given by 
\[
\omega_{ij}^\lambda (u^\eta_{kl}) = \delta_{\lambda\eta} \delta_{ik} \delta_{jl} 
\]
span the space $ \CD(G) $, which can be viewed as the dense $ * $-subalgebra of $ C_0(\hat{G}) $ given by the algebraic direct sum 
of matrix algebras $ \IK(\CH_\lambda) $ for $ \lambda \in \Irr(G) $. Moreover, in this picture 
the elements $ \omega_{ij}^\lambda $ are matrix units in $ \IK(\CH_\lambda) $, that is, 
\[
\omega^\lambda_{ij} \omega^\eta_{kl} = \delta_{\lambda \eta} \delta_{jk} \omega^{\lambda}_{il} 
\]

Let $(A,\alpha) $ be a $G$-\cstar-algebra. Recall from Remark \ref{spectralsub} that $ A $ becomes a right $ \CD(G) $-module 
with the action 
\[
a \cdot \omega = (\omega \otimes \id) \circ \alpha(a). 
\] 
It is crucial for our purposes that such module structures also exist on $ \ell^\infty(A) $ and $ A_\infty $. 
Indeed, note that applying $ \omega \otimes \id $ in each component we obtain slice maps 
$ \ell^\infty(C^\red(G) \otimes A) \to \ell^\infty(A) $ and $ (C^\red(G) \otimes A)_\infty \to A_\infty $, 
which, by slight abuse of notation, will again denoted by $ \omega \otimes \id $ in the sequel. 
We will also continue to use the notation $ a \cdot \omega $ for the module structures obtained in this way.

In analogy with the constructions in Remark \ref{spectralsub} we shall now define spectral 
subspaces of $ \ell^\infty(A) $ and $ A_\infty $, and use this to define corresponding continuous parts. Although the settings differ somewhat, this is similar to Kishimoto's definition of equicontinuous sequences for flows, cf.\ \cite{Kishimoto96, Kishimoto03}.

\begin{defi} \label{defspectralcont}
Let $G$ be a coexact compact quantum group and let $ (A,\alpha) $ be a $G$-\cstar-algebra. 
The spectral subspaces of $ \ell^\infty(A) $ and $ A_\infty $ with respect to $ \alpha $ are defined by 
\[
\CS(\ell^\infty(A)) = \ell^\infty(A) \cdot \CD(G), \qquad \CS(A_\infty) = A_\infty \cdot \CD(G), 
\] 
respectively. The continuous parts of $ \ell^\infty(A) $ and $ A_\infty $ with respect to $\alpha$ are defined by 
\[
\ell^{\infty, \alpha}(A) = [\CS(\ell^\infty(A))] \subset \ell^\infty(A), \qquad A_{\infty, \alpha} = [\CS(A_\infty)] \subset A_\infty, 
\]
respectively. 
\end{defi}

At this point it is not immediately obvious that the subspaces in Definition \ref{defspectralcont} are closed 
under multiplication. We will show this further below.

\begin{lemma} \label{spectralsurjective}
The canonical map $ \CS(\ell^\infty(A)) \to \CS(A_\infty) $ is surjective. 
\end{lemma}
\begin{proof} Let $ x \in \CS(A_\infty) $. Then we can write $ x = x \cdot p = (p \otimes \id) \alpha_\infty(x) $ for some 
finite rank idempotent $ p \in \CD(G) $. If $ \tilde{x} \in \ell^\infty(A) $ is any lift of $ x $, then 
$ \tilde{x} \cdot p $ is a lift of $ x $ as well, which in addition is contained in $ \CS(\ell^\infty(A)) $.
\end{proof} 

It follows from Lemma \ref{spectralsurjective} that the canonical map $\ell^{\infty, \alpha}(A) \to A_{\infty,\alpha}$ is surjective. 

\begin{prop} \label{descr cont part}
Let $G$ be a coexact compact quantum group and let $ (A,\alpha) $ be a $G$-\cstar-algebra. Then we have 
\[
\begin{array}{lcl}
\CS(\ell^\infty(A)) &=& \{x \in \ell^\infty(A) \mid \alpha^\infty(x) \in \CO(G) \odot \ell^\infty(A) \}; \\
\CS(A_\infty) &=& \{x \in A_\infty \mid \alpha_\infty(x) \in \CO(G) \odot A_\infty \}. 
\end{array}
\]
Moreover, both $ \ell^{\infty, \alpha}(A) $ and $ A_{\infty, \alpha} $ are $G$-\cstar-algebras in a canonical way. 
\end{prop} 
\begin{proof}
Let us consider first the assertions for $ \CS(\ell^\infty(A)) $. By construction, for an element $ x \in \CS(\ell^\infty(A)) $ satisfying 
$\alpha^\infty(x) \in \CO(G) \odot \ell^\infty(A)$ we can write
\[
\alpha^\infty(x) = \sum_{\lambda \in F} \sum_{i,j} u^\lambda_{ij} \otimes x^\lambda_{ij} 
\] 
for some finite set $ F \subset \Irr(G) $ and elements $ x^\lambda_{ij} \in \ell^\infty(A) $. 
By the definition of the counit $ \epsilon: \CO(G) \to \mathbb{C} $, we see that 
applying $ \omega = \sum_{\lambda \in F} \sum_{i = 1}^{\dim(\lambda)} \omega^\lambda_{ii} \in \CD(G) \subset C^\red(G)^* $ in the first tensor factor gives 
\[
x \cdot \omega = (\omega \otimes \id) \circ \alpha^\infty(x) = (\epsilon \otimes \id) \circ \alpha^\infty(x) = x.
\] 
This means $ \{x \in \ell^\infty(A) \mid \alpha^\infty(x) \in \CO(G) \odot \ell^\infty(A) \} \subset \CS(\ell^\infty(A)) $. 
 
Conversely, write $ x \in \CS(\ell^\infty(A)) $ as a finite sum $ x = \sum_i y^i \cdot \omega^i $ 
of some elements $ \omega^i \in \CD(G) $ and $ y^i \in \ell^\infty(A) $. We may assume without loss of generality 
that each $ \omega^i $ is contained in $ \IK(\CH_{\lambda_i}) $ for some $ \lambda_i \in \Irr(G) $. 
Let $ F \subset \Irr(G) $ be the finite subset consisting of all $ \lambda_j $. 
Writing $ x = (x_n)_{n \in \mathbb{N}} $ and $ y = (y_n)_{n \in \mathbb{N}} $ 
this means that each $ x_n = \sum_i y^i_n \cdot \omega^i $ is contained in the 
subspace $ \CS(A)_F = A \cdot \CO(G)_F \subset A \cdot \CO(G) $, where $ \CO(G)_F = \sum_i \CO(G)_{\lambda_i} $. 
Since $ \CO(G)_F $ is finite dimensional, we conclude
\[
\alpha^\infty(x) \in \CO(G) \odot \ell^\infty(A). 
\]
Here we use that the construction of $\ell^\infty $-products is compatible with tensoring by finite dimensional spaces. 
We conclude $ \CS(\ell^\infty(A)) \subset \{x \in \ell^\infty(A) \mid \alpha^\infty(x) \in \CO(G) \odot \ell^\infty(A) \} $.
 
The corresponding assertion for $ \CS(A_\infty) $ is obtained in a similar way. According to Lemma \ref{spectralsurjective}, we know 
that $ x \in \CS(A_\infty) $ is represented by an element $ \tilde{x} \in \CS(\ell^\infty(A)) $, so the above argument 
shows $ \alpha_\infty(x) \in \{z \in A_\infty \mid \alpha_\infty(z) \in \CO(G) \odot A_\infty \} $ by construction 
of $ \alpha_\infty $. Conversely, if $ x $ satisfies $ \alpha_\infty(x) \in \CO(G) \odot A_\infty $ 
then the same argument as in the case $ \ell^\infty(A) $ above shows $ x \in \CS(A_\infty) $.

As a consequence of these considerations we obtain in particular that $ \CS(\ell^\infty(A)) $ and $ \CS(A_\infty) $ 
are $ * $-algebras, and hence $ \ell^{\infty, \alpha}(A) \subset \ell^\infty(A) $and $ A_{\infty, \alpha} \subset A_\infty $ 
are \cstar-subalgebras. 

It remains to show that these \cstar-algebras are $G$-\cstar-algebras in a canonical way, with coactions induced 
by $ \alpha^\infty $ and $ \alpha_\infty $, respectively. 

Let us again first consider the case $ \ell^{\infty,\alpha}(A) $. From coassociativity of $ \alpha $ we obtain that 
$ \alpha^\infty $ maps $ \CS(\ell^\infty(A)) $ to $ \CO(G) \odot \CS(\ell^\infty(A)) $. Therefore 
it induces a $ * $-homomorphism $ \ell^{\infty, \alpha}(A) \to C^\red(G) \otimes \ell^{\infty, \alpha}(A) 
\subset C^\red(G) \otimes \ell^\infty(A) $, which we will again denote by $ \alpha^\infty $. 
Injectivity of the latter map is clear. Similarly, the coaction 
identity $ (\id \otimes \alpha^\infty) \alpha^\infty = (\Delta \otimes \id) \alpha^\infty $ follows immediately from the coaction 
identity for $ \alpha $. For the density condition note that we can write $ \eins \otimes x = S(x_{(-2)}) x_{(-1)} \otimes x_{(0)} $ 
for $ x \in \CS(\ell^\infty(A)) $, using the Hopf algebra structure of $ \CO(G) $, and the 
Sweedler notation $ \alpha^\infty(x) = x_{(-1)} \otimes x_{(0)} $. Hence 
\[ 
(\CO(G) \odot \eins) \alpha^\infty(\CS(\ell^\infty(A))) = \CO(G) \odot \CS(\ell^\infty(A)), 
\] 
which implies $ [(C^\red(G) \otimes \eins) \alpha^\infty(\ell^{\infty, \alpha}(A))] 
= C^\red(G) \otimes \alpha^\infty(\ell^{\infty, \alpha}(A)) $ upon taking closures.

The case $ A_{\infty, \alpha} $ is analogous. The considerations for $ \S(\ell^\infty(A)) $ above 
and Lemma \ref{spectralsurjective} imply that $ \alpha_\infty $ maps 
$ \CS(A_\infty) $ to $ \CO(G) \odot \CS(A_\infty) $. Therefore 
it induces a $ * $-homomorphism $ A_{\infty, \alpha} \to C^\red(G) \otimes A_{\infty, \alpha}
\subset C^\red(G) \otimes A_\infty $, which we will again denote by $ \alpha_\infty $. 
Coassociativity and density conditions are inherited from the corresponding properties 
of the coaction $ \alpha^\infty: \ell^{\infty, \alpha}(A) \rightarrow C^\red(G) \otimes \ell^{\infty, \alpha}(A) $. 
\end{proof}

Using Proposition \ref{descr cont part} we obtain an alternative way to describe the notion 
of equivariance introduced in Definition \ref{equiv sequence alg}

\begin{prop}
Let $G$ be a coexact compact quantum group and let $(A,\alpha)$ and $(B, \beta)$ be $G$-\cstar-algebras. For a $*$-homomorphism 
$\phi: A \to B_\infty$, the following are equivalent:
\begin{itemize}
\item[a)] $\phi:A \to B_\infty$ is $G$-equivariant;
\item[b)] $\phi(a \cdot \omega) = \phi(a) \cdot \omega$ for all $\omega \in \CD(G)$ and $a \in A$;
\item[c)] $\phi(A) \subset B_{\infty,\beta}$ and $\phi:(A,\alpha) \to (B_{\infty,\beta},\beta_\infty)$ is $G$-equivariant in the usual sense.
\end{itemize}
\end{prop}
\begin{proof}
$a) \Rightarrow b):$ Let $ \omega \in \CD(G)$ and $a \in A$, and recall $ \CD(G) \subset C^\red(G)^* $. 
Applying $ \omega \otimes \id $ to both sides of the 
equality $(\id \otimes \phi) \circ \alpha(a) = \beta_\infty \circ \phi(a) $, we obtain 
\[
\begin{array}{lclcl}
\phi(a \cdot \omega) &=& (\id \otimes \phi)(\omega \otimes \id)\alpha(a) \\
&=& (\omega \otimes \id)(\id \otimes \phi)\alpha(a) \\
&=& (\omega \otimes \id) \beta_\infty \circ \phi(a) \\
&=& \phi(a) \cdot \omega.
\end{array}
\]

$b) \Rightarrow c): $ As $\phi(a \cdot \omega) = \phi(a) \cdot \omega $ for all $ \omega \in \CD(G)$ and $a \in A$, 
it follows that $\phi(\CS(A)) \subset \CS(B_\infty)$. Hence, $ \beta_\infty \circ \phi $ maps 
$\CS(A)$ into $\CO(G) \odot \CS(B_\infty)$. For $a \in \CS(A)$ and $ \omega \in \CD(G)$ we therefore compute
\[
(\omega \otimes \id)\beta_\infty\phi(a) = \phi(a) \cdot \omega = \phi(a \cdot \omega) = (\omega \otimes \id)(\id \otimes \phi)\alpha(a).
\]
It is now straightforward to check that $\beta_\infty\circ \phi(a) = (\id \otimes \phi)\circ \alpha(a)$ 
for all $a \in \CS(A)$. As $\CS(A) \subset A$ is dense, we conclude that $\phi: A\to B_{\infty, \beta}$ is $G$-equivariant.
 
$c) \Rightarrow a):$ This implication follows immediately from the definitions.
\end{proof}

We shall now define crossed products of induced actions on sequence algebras. 

\begin{defi} \label{deficrossedpr}
Let $G$ be a coexact compact quantum group and let $(A,\alpha)$ be a $G$-\cstar-algebra. We define 
\[ 
G \ltimes_{\alpha_\infty, \red} A_\infty = G\ltimes_{\alpha_\infty, \red} A_{\infty, \alpha}, 
\] 
that is, $ G \ltimes_{\alpha_\infty, \red} A_\infty $ is defined to be the crossed product 
of the continuous part of $ A_\infty $ with respect to the 
coaction $ \alpha_\infty: A_{\infty, \alpha} \rightarrow C^\red(G) \otimes A_{\infty, \alpha} $. 
In a similar way we define 
\[ 
G\ltimes_{\alpha^\infty, \red} \ell^\infty(A) = G \ltimes_{\alpha^\infty, \red} \ell^{\infty, \alpha}(A). 
\] 
\end{defi} 

\begin{rem} 
The notation introduced in Definition \ref{deficrossedpr} will allow us to unify our exposition of several results in subsequent sections. 
Remark that the crossed products $ G \ltimes_{\alpha_\infty, \red} A_\infty $ and $ G \ltimes_{\alpha^\infty, \red} \ell^\infty(A) $ 
carry honest $ \check{G} $-\cstar-algebra structures given by the dual actions. 
\end{rem}

At a few points we will need a notion of equivariance for $ * $-homomorphisms with 
target $ D_{\infty,B} $ or $ M(D_{\infty,B}) $. 

Let $G$ be a coexact compact quantum group and $ (B, \beta) $ be a $G$-\cstar-algebra. 
Note that nondegeneracy of the $*$-homomorphism $\beta:B \to C^\red(G) \otimes B$ implies that 
\[
\beta_\infty(D_{\infty,B}) \subset D_{\infty, C^\red(G) \otimes B} = (C^\red(G) \otimes B) (C^\red(G) \otimes B)_\infty (C^\red(G) \otimes B)
\]
is a nondegenerate \cstar-subalgebra. Hence $ \beta_\infty $ induces a $ * $-homomorphism 
\[ 
M(D_{\infty, \beta}) \rightarrow M(D_{\infty, C^\red(G) \otimes B}), 
\]
which we will again denote by $ \beta_\infty $.

\begin{defi} \label{defequivmdinfty}
Let $G$ be a coexact compact quantum group and let $ (A, \alpha), (B, \beta) $ be $G$-\cstar-algebras.
A $*$-homomorphism $\phi: A \to M(D_{\infty,B})$ is called $G$-equivariant if 
$\beta_\infty \circ \phi = (\id \otimes \phi) \circ \alpha$, where both sides are viewed as maps from 
$ A $ into $ M(D_{\infty, C^\red(G) \otimes B}) $. 
If $\phi$ is $G$-equivariant, then we also write $\phi:(A,\alpha) \to (M(D_{\infty,B}),\beta_\infty)$.
\end{defi} 

\begin{rem} \label{prop equiv morph}
Note in particular that if $ \phi: A \to M(D_{\infty,B})$ is $G$-equivariant in the sense of Definition \ref{defequivmdinfty} 
then $\beta_\infty \circ \phi(A) \subset C^\red(G) \otimes M(D_{\infty,B}) \subset M(D_{\infty, C^\red(G) \otimes B})$. 

It is immediate from the definitions that a $*$-homomorphism $\phi: A \to D_{\infty,B}$ is $G$-equivariant as a $*$-homomorphism 
$A \to B_\infty$ if and only if it is $G$-equivariant as a $*$-homomorphism $A \to M(D_{\infty,B})$. 
\end{rem}


\section{Equivariantly sequentially split $*$-homomorphisms}

In this section we discuss the notion of sequentially split $*$-homomorphisms between $G$-\cstar-algebras, which was studied 
in \cite{BarlakSzabo15} in the case of actions by groups. 

\begin{defi}[cf.~{\cite[2.1, 3.3]{BarlakSzabo15}}]
Let $G$ be a quantum group which is either discrete or compact and coexact. Moreover 
let $ (A, \alpha), (B, \beta) $ be $G$-\cstar-algebras. We say that an equivariant $*$-homomorphism $\phi: (A,\alpha)\to (B,\beta)$ is equivariantly sequentially split if there exists a commutative diagram of $G$-equivariant $*$-homomorphisms of the form
\[
\xymatrix{
(A,\alpha) \ar[dr]_\phi \ar[rr] && (A_\infty, \alpha_\infty) \\
& (B,\beta) \ar[ur]_\psi &
}
\]
where the horizontal map is the standard embedding. If $\psi: (B, \beta) \to (A_\infty, \alpha_\infty) $ is an 
equivariant $*$-homomorphism fitting into the above diagram, then we say that $\psi$ is an equivariant approximate left-inverse for $\phi$.
\end{defi}

An important feature of the theory of sequentially split $*$-homomorphisms is that it is compatible with forming crossed product \cstar-algebras. The proof makes use of the following fact.

\begin{lemma} \label{nat morph cross prod}
Let $G$ be a quantum group which is either discrete and exact or compact and coexact. Moreover let $(A, \alpha)$ be 
a $G$-\cstar-algebra. Then there exists a $\check{G}$-equivariant $*$-homomorphism 
$G \ltimes_{\alpha_\infty, \red} A_\infty \to (G \ltimes_{\alpha, \red} A)_\infty$, compatible with the natural inclusions 
of $G \ltimes_{\alpha, \red} A$ on both sides.
\end{lemma}
\begin{proof}
Assume first that $G$ is discrete and exact. Since taking reduced crossed products with $ G $ is exact, the 
canonical map $ G \ltimes_{\alpha^\infty, \red} \ell^\infty(A) \rightarrow \ell^\infty(G \ltimes_{\alpha,\red} A) $ induces a 
commutative diagram with exact rows
\[
\xymatrix{
0 \ar[r] & G \ltimes_{c_0(\alpha), \red} c_0(A) \ar[r] \ar[d]^\cong & G \ltimes_{\alpha^\infty, \red} \ell^\infty(A) \ar[r] \ar[d] 
& G \ltimes_{\alpha_\infty, \red} A_\infty \ar[r] \ar[d] & 0 \\
0 \ar[r] & c_0(G \ltimes_{\alpha, \red} A) \ar[r] & \ell^\infty(G \ltimes_{\alpha, \red} A) \ar[r] 
& (G \ltimes_{\alpha, \red} A)_\infty \ar[r] & 0. 
}
\]
Here $ c_0(\alpha): c_0(A) \rightarrow M_{C_0(G)}(C_0(G) \otimes c_0(A)) $ denotes the restriction 
of $ \alpha^\infty $ to $ c_0(A) $. 

It is clear from the construction that the $*$-homomorphism 
$G \ltimes_{\alpha_\infty, \red} A_\infty \to (G \ltimes_{\alpha, \red} A)_\infty$ is $\check{G}$-equivariant and compatible with the canonical inclusions of $G \ltimes_{\alpha, \red} A$.

Assume now that $G$ is compact and coexact. Let us abbreviate $ \IK = \IK(L^2(G))$. 
Then the canonical map $ \IK \otimes \ell^\infty(A) \rightarrow \ell^\infty( \IK \otimes A) $ induces the following commutative diagram 
with exact rows
\[
\xymatrix{
0 \ar[r] & \IK \otimes c_0(A) \ar[r] \ar[d]^\cong & \IK \otimes \ell^\infty(A) \ar@{^(->}[d] \ar[r] & \IK \otimes A_\infty \ar[r] \ar@{^(->}[d] \ar[r] & 0\\
0 \ar[r] & c_0(\IK \otimes A) \ar[r] & \ell^\infty( \IK \otimes A) \ar[r] & (\IK \otimes A)_\infty \ar[r] & 0
}
\]
The middle vertical arrow restricts to a $*$-homomorphism $ G \ltimes_{\alpha^\infty, \red} \ell^\infty(A) 
\into \ell^\infty(G \ltimes_{\alpha, \red} A)$. By Lemma \ref{spectralsurjective}, the canonical $\check{G}$-equivariant 
map $G \rtimes_{\alpha^\infty, \red} \ell^\infty(A) \to G \ltimes_{\alpha_\infty, \red} A_\infty $ is surjective. 
Observe moreover that the canonical surjection $\ell^{\infty}(\IK \otimes A) \to (\IK \otimes A)_\infty$ restricts to the canonical 
surjection $\ell^{\infty}(G \ltimes_\red A) \to (G \ltimes_\red A)_\infty$. It follows that the 
embedding $\IK \otimes A_\infty \into (\IK \otimes A)_\infty$ restricts to an embedding 
$G \ltimes_{\alpha_\infty, \red} A_\infty \into (G \ltimes_{\alpha, \red} A)_\infty$. This map is clearly compatible with the 
canonical embeddings of $G \ltimes_{\alpha, \red} A$. Moreover, 
as $G \ltimes_{\alpha^\infty, \red} \ell^\infty(A) \to (G \ltimes_{\alpha, \red} A)_\infty$ is $\check{G}$-equivariant, this also 
holds for $G \ltimes_{\alpha_\infty, \red} A_\infty \to (G \ltimes_{\alpha, \red} A)_\infty$. This finishes the proof.
\end{proof}

\begin{prop} \label{form-crossed-products}
Let $G$ be quantum group which is either discrete and exact or compact and coexact. Moreover let $(A, \alpha)$ and $(B, \beta)$ 
be $G$-\cstar-algebras. Assume that $\phi:(A,\alpha)\to (B,\beta)$ is an equivariantly sequentially split $*$-homomorphism.

Then the induced $*$-homomorphism $ G \ltimes_\red \phi: G\ltimes_{\alpha, \red} A \to G \ltimes_{\beta, \red} B$ between the crossed products is $\check{G}$-equivariantly sequentially split.
\end{prop}
\begin{proof}
Let $\psi:(B,\beta)\to (A_{\infty},\alpha_\infty)$ be an approximate left-inverse for $\phi$. Passing to crossed products, we obtain a commutative diagram of $\check{G}$-equivariant $*$-homomorphisms
\[
\xymatrix{
(G \ltimes_{\alpha, \red} A,\check{\alpha}) \ar[rr] \ar[rd]_{G\ltimes_\red \phi} & & (G \ltimes_{\alpha_\infty, \red} A_\infty, \check{\gamma}) \\
 & (G \ltimes_{\beta, \red} B,\check{\beta}) \ar[ru]_{G \ltimes_\red \psi}
}
\]
where $ \gamma = \alpha_\infty $.  
Composing $G \ltimes_\red \psi$ with the $\check{G}$-equivariant $*$-homomorphism 
$G \ltimes_{\alpha_\infty, \red} A_\infty \to (G \ltimes_{\alpha, \red} A)_\infty$ from Lemma \ref{nat morph cross prod} yields an equivariant approximate left-inverse for $G\ltimes_\red \phi$.
\end{proof}

Let us next recall the definition of the fixed point algebra of an action of a compact quantum group. 

\begin{defi}
Let $G$ be a compact quantum group. For a $G$-\cstar-algebra $ (A,\alpha) $ we denote by  
\[
A^\alpha = \{a \in A \mid \alpha(a) = \eins \otimes a\} \subset A 
\]
the \cstar-subalgebra of fixed points. 
\end{defi}

\begin{lemma} \label{fixed-point-alg}
Let $G$ be a coexact compact quantum group and let $ (A,\alpha) $ be a $ G $-\cstar-algebra. 
Then the canonical inclusion $ (A^\alpha)_\infty \rightarrow (A_{\infty, \alpha})^{\alpha_\infty} $ is an isomorphism.
\end{lemma}
\begin{proof}
Let $ a \in (A_{\infty, \alpha})^{\alpha_\infty} $ be represented by $ (a_n)_{n \in \mathbb{N}} \in \ell^\infty(A) $. 
Then 
\[
\lim_{n \rightarrow \infty} \|\alpha(a_n) - \eins \otimes a_n \| = 0 
\]
by the fixed point condition. Applying the Haar state $ \phi: C^\red(G) \to \mathbb{C} $ in the first tensor factor 
gives 
\[
\lim_{n \rightarrow \infty} \|(\phi \otimes \id)\circ\alpha(a_n) - a_n \| = 0 
\]
Since $ (\phi \otimes \id) \circ \alpha $ maps $ A $ into $ A^\alpha $ we conclude that $ a $ is represented by an 
element of $ \ell^\infty(A^\alpha) $. 
\end{proof} 

As we show next, a naturality property as in Proposition \ref{form-crossed-products} also holds for fixed point algebras of actions of 
compact quantum groups. 

\begin{prop} \label{form-fixed-point-algs}
Let $G$ be a coexact compact quantum group, and let $ (A,\alpha), (B, \beta) $ be $ G $-\cstar-algebras. 
Assume that $\phi:(A,\alpha)\to (B,\beta)$ is a $G$-equivariantly sequentially split $*$-homomorphism. 

Then the induced $*$-homomorphism $\phi:A^\alpha\to B^\beta$ is a sequentially split.
\end{prop}
\begin{proof}
Let $\psi:(B,\beta)\to (A_\infty, \alpha_\infty) $ be an equivariant approximate left-inverse for $\phi$. 
By equivariance of $\psi$, we have 
\[
\alpha_\infty \circ \psi(b) = (\id \otimes \psi)\circ \beta(b) = \eins \otimes \psi(b)
\]
for all $ b \in B^\beta $. Hence $ \psi $ maps $ B^\beta $ into $ (A_{\infty, \alpha})^{\alpha_\infty} $. According 
to Lemma \ref{fixed-point-alg} the latter identifies with $ (A^\alpha)_\infty $, and therefore 
$ \psi: B^\beta \rightarrow (A^\alpha)_\infty $ is an approximate left-inverse for $ \phi: A^\alpha \to B^\beta $. 
\end{proof}

The following stability result is an important feature for the theory of sequentially split $*$-homomorphisms.

\begin{prop} \label{proptensorwithcompacts}
Let $ G $ be a quantum group which is either compact and coexact or discrete and exact. Moreover 
let $ \phi: (A, \alpha) \to (B, \beta) $ be a nondegenerate $ G $-equivariant $ * $-homomorphism 
between $ G $-\cstar-algebras. 

Then $ \phi $ is $ G $-equivariantly sequentially split 
if and only if $ \id \otimes \phi: (\IK(L^2(G)) \otimes A, \alpha_\IK) \rightarrow (\IK(L^2(G)) \otimes B, \beta_\IK) $ 
is $ G $-equivariantly sequentially split. 
\end{prop} 
\begin{proof}
Let us first consider the case that $ G $ is compact and coexact. 

Assume that $ \phi $ is $ G $-equivariantly sequentially split, and let $ \psi: (B, \beta) \rightarrow (A_\infty, \alpha_\infty) $ 
be a $ G $-equivariant approximate left inverse of $ \phi $. 
Then the map $ \id \otimes \psi: \IK(L^2(G)) \otimes B \rightarrow \IK(L^2(G)) \otimes A_\infty \subset (\IK(L^2(G)) \otimes A)_\infty $ 
is $ G $-equivariant, and yields an equivariant approximate left-inverse for $ \id \otimes \phi $.

Conversely, assume that $ \id \otimes \phi: (\IK(L^2(G)) \otimes A, \alpha_\IK) \rightarrow (\IK(L^2(G)) \otimes B, \beta_\IK) $ 
is $ G $-equivariantly sequentially split. Let $\Psi: (\IK(L^2(G)) \otimes B, \beta_\IK) 
\to ((\IK(L^2(G)) \otimes A)_\infty, (\alpha_\IK)_\infty)$ be a $G$-equivariant approximate left-inverse. As $\phi$ is 
assumed to be non-degenerate, the image of $ \Psi $ is contained 
in $D_{\infty,\IK(L^2(G)) \otimes A}$. Using the isomorphism $D_{\infty,\IK(L^2(G)) \otimes A} \cong \IK(L^2(G)) \otimes D_{\infty,A}$ from Lemma \ref{lemmaDomegastable}, we see that $\Psi$ defines a nondegenerate $ * $-homomorphism 
from $\IK(L^2(G)) \otimes B$ into $\IK(L^2(G)) \otimes D_{\infty,A}$. Let us denote the extension 
$ M(\IK(L^2(G)) \otimes B) \to M(\IK(L^2(G)) \otimes D_{\infty,A})$ to multiplier algebras again by $ \Psi $.
 
We shall write $ \Psi_B $ for the restriction of $ \Psi $ to $ B \cong \eins \otimes B \subset M(\IK(L^2(G)) \otimes B) $. 
Then $ \Psi_B: B \rightarrow M(D_{\infty,\IK(L^2(G) \otimes A}) $ is a $ * $-homomorphism whose image is 
contained in the relative commutant of $ \IK(L^2(G)) \otimes \eins $. According to \cite[1.8]{BarlakSzabo15}, 
its image $ \im(\Psi_B) $ is therefore contained in $ \eins \otimes M(D_{\infty,A}) $. Using again nondegeneracy of $ \phi $ we 
see that $ \im(\Psi_B) $ is in fact contained in $ \eins \otimes D_{\infty,A} $. 
From these observations and the sequential split property we conclude that $ \Psi $ can be written in the 
form $ \Psi = \id_{\IK(L^2(G))} \otimes \psi $ for a non-degenerate $ * $-homomorphism $ \psi: B \rightarrow D_{\infty,A} $. It is easy to check that $ \psi $ is an approximate left-inverse for $ \phi $.

We claim that $ \psi: B \rightarrow D_{\infty,A} $ is $ G $-equivariant. 
For this consider a simple tensor $ T \otimes b \in \IK(L^2(G)) \otimes B $ and compute 
\[
\begin{array}{lcl}
(\id \otimes \Psi)\beta_\IK(T \otimes b) & =& (\id \otimes \Psi)(X^*_{12} (\eins \otimes T \otimes \eins) \beta(b)_{13} X_{12}) \\ 
& = & (\id \otimes \id \otimes \psi)(X^*_{12} (\eins \otimes T \otimes \eins) \beta(b)_{13} X_{12}) \\ 
& = & X^*_{12} (\eins \otimes T \otimes \eins) ((\id \otimes \psi)\beta(b))_{13} X_{12} 
\end{array}
\]
and 
\[
(\alpha_\IK)_\infty(\Psi(T \otimes b)) = (\alpha_\IK)_\infty(T \otimes \psi(b)) 
= X^*_{12} (\eins \otimes T \otimes \eins) \alpha_\infty(\psi(b))_{13} X_{12}. 
\]
Here all expressions are viewed as elements of $ (C^\red(G) \otimes \IK(L^2(G)) \otimes A)_\infty $. Equivariance of $ \Psi $ 
means that the above expressions are equal. We conclude $ (\id \otimes \psi)\beta(b) = \alpha_\infty(\psi(b)) $ for all $b \in B$ 
as desired. 

In the case that $ G $ is discrete and exact we can follow the above arguments almost word by word, in this case the situation 
is even slightly easier since all algebras involved are honest $ G$-\cstar-algebras. 
\end{proof}

\begin{prop} \label{general-duality}
Let $ G $ be a quantum group which is either compact and coexact or discrete and exact. 
Moreover assume that $(A, \alpha), (B,\beta)$ are
separable $G$-\cstar-algebras, and let $\phi: (A,\alpha)\to (B,\beta)$ be a non-degenerate equivariant $*$-homomorphism. 

Then $\phi$ is $G$-equivariantly sequentially split if and only if
\[
\check{\phi} = G \ltimes_\red \phi: (G \ltimes_{\alpha, \red} A, \check{\alpha}) \to (G \ltimes_{\beta, \red} B, \check{\beta})
\]
is $\check{G}$-equivariantly sequentially split.
\end{prop}
\begin{proof}
If $\phi$ is $G$-equivariantly sequentially split, then Proposition \ref{form-crossed-products} shows  
that $\check{\phi}$ is $\check{G}$-equivariantly sequentially split. 
On the other hand, if $\check{\phi}$ is $\check{G}$-equivariantly sequentially split, then 
Proposition \ref{form-crossed-products} implies that $\check{\check{\phi}}$ is $G$-equivariantly sequentially split. Under 
the $G$-equivariant isomorphism given by Takesaki-Takai duality, see 
Theorem \ref{TTduality}, the map $\check{\check{\phi}}$ corresponds 
to $\id_{\IK} \otimes \phi: (\IK \otimes A, \alpha_\IK) \to (\IK \otimes B, \beta_\IK) $. Here we abbreviate $ \IK = \IK(L^2(G)) $. 
Given an equivariant approximate left-inverse $ (\check{G} \ltimes_{\check{\beta}, \red} G \ltimes_{\beta, \red} B,\check{\check{\beta}}) 
\to ((\check{G} \ltimes_{\check{\alpha}, \red} G \ltimes_{\alpha, \red} A)_\infty,\check{\check{\alpha}}_\infty)$ for $\check{\check{\phi}}$, componentwise application of Takesaki-Takai duality yields a $G$-equivariant approximate 
left-inverse $(\IK \otimes B, \beta_\IK) \to ((\IK \otimes A)_\infty, (\alpha_\IK)_\infty)$ for $\id_\IK \otimes \phi$. 
Hence, $\id_\IK \otimes \phi$ is $G$-equivariantly sequentially split. The claim now follows from 
Proposition \ref{proptensorwithcompacts}.
\end{proof}


\section{The Rokhlin property and approximate representability}

In this section we introduce the key notions of this paper, namely the spatial Rokhlin property and 
spatial approximate representability for actions of quantum groups. Moreover we prove that these notions are dual to each 
other. 

\subsection{The spatial Rokhlin property} 

Let us start by defining the spatial Rokhlin property. 
\begin{defi} \label{def:rp}
Let $G$ be a coexact compact quantum group and $ (A, \alpha) $ a separable $G$-\cstar-algebra. We say that $\alpha$ has the spatial Rokhlin property if the second-factor embedding
\[
\iota_\alpha=\eins\boxtimes\id_A: (A,\alpha) \to \big( C^\red(G)\boxtimes A, \Delta\boxtimes\alpha \big)
\]
is $G$-equivariantly sequentially split.
\end{defi}

\begin{rem} \label{spatialremark1}
Definition \ref{def:rp} is indeed a generalization of the classical notion of the Rokhlin property, see \cite[4.3]{BarlakSzabo15}. 
Indeed, for a classical compact group $ G $, the braided tensor product $ C(G) \boxtimes A $ agrees with the ordinary tensor product 
$ C(G) \otimes A $. The term \emph{spatial} in Definition \ref{def:rp} refers to the fact that we have chosen to work with 
minimal (braided) tensor products; we will comment further on the implications of this choice in Remark \ref{spatialremark2} below. 
\end{rem} 

In special cases the Rokhlin property can be recast in the following way. Recall that if $ G $ is a compact 
quantum group and $ (A,\alpha) $ a $G$-\cstar-algebra we write $ \CS(A) $ for the spectral subalgebra of $A$. 
We shall use the Sweedler notation $ \alpha(a) = a_{(-1)} \otimes a_{(0)} $ for the coaction $ \alpha: \CS(A) \to \CO(G) \odot \CS(A) $. 

\begin{prop} \label{char:rp}
Let $G$ be a coexact compact quantum group and $ (A, \alpha) $ a separable $G$-\cstar-algebra. 
\begin{enumerate}[label=\textup{(\alph*)},leftmargin=*] 
\item[a)] If $\alpha$ has the spatial Rokhlin property, then there exists a unital and $G$-equivariant 
$*$-homomorphism $\kappa: (C^\red(G),\Delta)\to \big( M(D_{\infty,A}), \alpha_\infty \big)$ satisfying
\[
a \kappa(f)=\kappa(a_{(-2)} f S(a_{(-1)})) a_{(0)}
\]
for all $f\in \CO(G)\subset C^\red(G)$ and $a\in\CS(A) \subset A$. Moreover we have $ \|\kappa(S(a_{(-1)})) a_{(0)} \| \leq \|a\| $ 
for all $ a \in \CS(A) $. 
\item[b)] Assume $G$ is coamenable. If a $*$-homomorphism $\kappa: C^\red(G) \rightarrow M(D_{\infty, A})$ 
as in $ a) $ exists such that $ \|\kappa(S(a_{(-1)})) a_{(0)} \| \leq \|a\| $ for all $ a \in \CS(A) $, 
then $\alpha$ has the spatial Rokhlin property. 
\end{enumerate}
\end{prop}
\begin{proof}
$ a) $ Assume first that $\alpha$ has the spatial Rokhlin property. Let
\[
\psi: \big(C^\red(G)\boxtimes A, \Delta\boxtimes\alpha \big) \to (A_\infty,\alpha_\infty)
\]
be an equivariant, approximate left-inverse for $\iota_\alpha$ as required by Definition \ref{def:rp}. Since $\iota_\alpha$ 
is non-degenerate, the image of this $*$-homomorphism is contained in $D_{\infty,\alpha}$, 
and $\psi: C^\red(G) \boxtimes A \rightarrow D_{\infty, A} $ is again nondegenerate. 
Let us also denote the unique strictly continuous extension of $ \psi $ to multipliers by the same letter, so that 
we have 
\[
\psi: M\big(C^\red(G)\boxtimes A \big) \to M(D_{\infty,A}).
\]
By equivariance of $\psi$, the unital $*$-homomorphism 
\[
\kappa=\psi\circ\iota_{C^\red(G)}: C^\red(G)\to M(D_{\infty,A})
\] 
is also equivariant. According to the definition of the braided tensor product, we have
\[
\iota_A(a) \iota_{C^\red(G)}(f) = \iota_{C^\red(G)}\big(a_{(-2)}f S(a_{(-1)})\big) \iota_A(a_{(0)})
\]
for all $f\in\CO(G)$ and $a\in\CS(A)$. The desired twisted commutation relation for $\kappa$ then follows 
by applying $\psi$. Moreover, the norm condition $ \|\kappa(S(a_{(-1)})) a_{(0)} \| \leq \|a\| $ 
is a consequence of the formula $ \kappa(S(a_{(-1)})) a_{(0)} = \psi(T_\alpha^{-1}(1 \otimes a)) $ for $ a \in \CS(A) $ 
and the fact that $ \psi $ and the isomorphism $ T_\alpha^{-1} $ from Lemma \ref{trivialization}
are $*$-homomorphisms between \cstar-algebras. 

$ b) $ Consider the map $ \iota: \CS(A) \rightarrow D_{\infty, A} $ given by $ \iota(a) = \kappa(S(a_{(-1)})) a_{(0)} $. Then 
the commutation relation for $ \kappa $ gives 
\[
\iota(a) \kappa(f) = \kappa(S(a_{(-3)}) a_{(-2)}f S(a_{(-1)})) a_{(0)} = \kappa(f) \iota(a) 
\]
for any $ f \in \CO(G) $, using the antipode relation for the Hopf algebra $ \CO(G) $. Using that $ S $ 
is antimultiplicative we therefore obtain 
\[
\begin{array}{lcl}
\iota(ab) & = & \kappa(S(a_{(-1)}b_{(-1)})) a_{(0)} b_{(0)} \\ 
& = & \kappa(S(b_{(-1)})) \kappa(S(a_{(-1)})) a_{(0)} b_{(0)} \\ 
& = & \kappa(S(b_{(-1)})) \iota(a) b_{(0)} \\ 
& = & \iota(a) \kappa(S(b_{(-1)})) b_{(0)} = \iota(a) \iota(b) 
\end{array}
\]
for all $ a,b \in \CS(A) $. Moreover, using $ S(h^*) = S^{-1}(h)^* $ for all $ h \in \CO(G) $ we have 
\[
\begin{array}{lcl}
\iota(a^*) = \kappa(S(a_{(-1)}^*)) a_{(0)}^* &=& \kappa(S^{-1}(a_{(-1)})^*) a_{(0)}^* \\
&=& (a_{(0)} \kappa(S^{-1}(a_{(-1)})))^* \\
&=& (\kappa(a_{(-2)} S^{-1}(a_{(-3)}) S(a_{(-1)})) a_{(0)})^* \\
&=& (\kappa(S(a_{(-1)})) a_{(0)})^* = \iota(a)^* 
\end{array}
\]
for $ a \in \CS(A) $. It follows that $ \iota $ is a $ * $-homomorphism. 
By assumption $ \iota $ is bounded, so that it extends to a $ * $-homomorphism $ \iota: A \rightarrow D_{\infty, A} $. 

Combining $ \kappa $ and $ \iota $ we obtain a $ * $-homomorphism $ \psi = \kappa \otimes \iota: C(G) \otimes A 
\cong C(G) \otimes_{\text{max}} A\rightarrow D_{\infty, A} $, using the universal property of the maximal tensor product and 
nuclearity of $ C(G) $. 
Since $ \kappa $ is equivariant one checks that $ \iota $ maps into the fixed point algebra of 
$ A_{\infty, \alpha} $, and together with equivariance of $ \kappa $ it follows that $ \psi: (C(G) \otimes A, \Delta \otimes \id) 
\rightarrow (D_{\infty, A}, \alpha_\infty) $ is $G$-equivariant. 
Using the isomorphism from Lemma \ref{trivialization}
we see that it defines an approximate left-inverse for $ \iota_A: A \to C(G) \boxtimes A \cong C(G) \otimes A $.  
\end{proof}

\begin{rem} 
Let us point out that the norm condition in part $ b) $ of Proposition \ref{char:rp} is automatically satisfied if $ G $ is a 
finite quantum group. Indeed, in this case the spectral subalgebra $ \CS(A) $ is equal to $ A $, and the claim follows 
from the fact that $ * $-homomorphisms between \cstar-algebras are contractive. 
It can also be shown that that the norm condition always holds if $ G $ is a classical compact group, but 
it seems unclear whether it is automatic in general. 
\end{rem}

Classically, a Rokhlin action of a compact group on an abelian \cstar-algebra $ C_0(X) $ induces a free action of $ G $ 
on $ X $. In the quantum case, an analogue of the notion of freeness has been formulated by Ellwood in \cite{Ellwood00}. 
Namely, an action $ \alpha: A \rightarrow C^\red(G) \otimes A $ of a compact quantum group $ G $ on a \cstar-algebra $ A $ is called 
free if $ [(\eins \otimes A) \alpha(A)] = C^\red(G) \otimes A $. It is shown in \cite[Theorem 2.9]{Ellwood00} that this generalizes the 
classical concept of freeness. 

Let us verify that the spatial Rokhlin property implies freeness also in the quantum case. 

\begin{prop} 
Let $ G $ be a coexact compact quantum group and let $ (A, \alpha) $ be a separable $G$-\cstar-algebra. 
If $\alpha$ has the spatial Rokhlin property, then it is free. 
\end{prop} 
\begin{proof}
Let $ \psi: C^\red(G) \boxtimes A \rightarrow A_\infty $ be an equivariant approximate left-inverse for the inclusion 
map $ \iota_A: A \to C^\red(G) \boxtimes A $. Notice that the action of $ G $ on $ (C^\red(G) \boxtimes A, \Delta \boxtimes \alpha) 
\cong (C^\red(G) \otimes A, \Delta \otimes \id) $ 
is free, so that 
\[ 
[(\Delta \boxtimes \alpha)(C^\red(G) \boxtimes A)(\eins \otimes C^\red(G) \boxtimes A)] = C^\red(G) \otimes (C^\red(G) \boxtimes A).
\]
In fact, for any $ f \in \CO(G) \subset C^\red(G) $ and $ a \in A $ 
we find finitely many elements $ x^i, y^i \in C^\red(G) \boxtimes A $ 
such that $ f \otimes (\eins \boxtimes a) = \sum_i (\Delta \boxtimes \alpha)(x^i) (\eins \otimes y^i) $ 
in $ C^\red(G) \otimes (C^\red(G) \boxtimes A) $. Applying $ \id \otimes \psi $ to this equality and using equivariance, we obtain 
\[
f \otimes \psi(\eins \boxtimes a) = f \otimes a = \sum_i \alpha_\infty \psi(x^i) (\eins \otimes \psi(y^i)) 
\]
in $ (C^\red(G) \otimes A)_\infty $. Consider lifts $ (\tilde{x}^i_n)_{n \in \mathbb{N}} $ 
and $ (\tilde{y}^i_n)_{n \in \mathbb{N}} $ in $ \ell^\infty(A) $ for $ \psi(x^i), \psi(y^i) $, respectively. 
Then 
\[
f \otimes a = \lim_{n\to\infty} \sum_i \alpha(\tilde{x}^i_n)(\eins \otimes \tilde{y}^i_n),  
\]
which implies $ C^\red(G) \odot A \subset [\alpha(A) (\eins \otimes A)] $, and hence 
also $ C^\red(G) \otimes A = [\alpha(A)(\eins \otimes A)] $.
\end{proof}

Let us point out that the Rokhlin property is strictly stronger than freeness; this is already the case classically. For instance, 
the antipodal action of $ G = \mathbb{Z}_2 $ on $ S^1 $ does not have the Rokhlin property. 

Here comes the first main result of this paper:

\begin{theorem}
Let $G$ be a coexact compact quantum group and let $ (A,\alpha) $ be a separable $G$-\cstar-algebra. If $\alpha$ has the spatial Rokhlin property, then the two canonical embeddings
\[
A^\alpha\into A \quad\text{and}\quad G \ltimes_{\alpha, \red} A \into \IK(L^2(G))\otimes A
\]
are sequentially split. In particular, if $A$ has any of the following properties, then so do the fixed-point algebra $A^\alpha$ 
and the crossed product $G \ltimes_{\alpha, \red} A$:
\begin{itemize}
\item being simple;
\item being nuclear and satisfying the UCT;
\item having finite nuclear dimension or decomposition rank;
\item absorbing a given strongly self-absorbing \cstar-algebra $\CD$.
\end{itemize}
\end{theorem}
\begin{proof}
Let $ \psi: C^\red(G) \boxtimes A \rightarrow A_\infty $ be an equivariant approximate left-inverse for the 
embedding $ \iota_A: A \to C^\red(G) \boxtimes A $. The resulting commutative diagram of equivariant $ * $-homomorphisms 
\[ 
\xymatrix{
A \ar[rr] \ar[dr] && A_\infty \\
& C^\red(G) \boxtimes A \ar[ru]_\psi & 
}
\] 
induces a commutative diagram of $ * $-homomorphisms 
\[ 
\xymatrix{
A^\alpha \ar[rr] \ar[dr] && (A^\alpha)_\infty \\
& A \ar[ru]_\psi & 
}
\] 
by Proposition \ref{form-fixed-point-algs}. Notice here 
that Lemma \ref{trivialization} implies that $ (C^\red(G) \boxtimes A)^{\Delta \boxtimes \alpha} \cong A $ in such a way 
that the canonical embeddings of $A^\alpha$ on both sides are compatible. 

For the statement about the crossed product $G \ltimes_{\alpha, \red} A$, observe that the spatial Rokhlin property for $ \alpha $ 
means that $ \alpha: A \rightarrow C^\red(G) \otimes A $ is sequentially split, taking into account the isomorphism  
from Lemma \ref{trivialization}. According to Proposition \ref{form-crossed-products}, it follows that
the map 
\[
G \ltimes_\red \alpha:G \ltimes_{\alpha, \red} A \into G \ltimes_{\Delta \otimes \id_A, \red} (C^\red(G) \otimes A) 
\cong (G \ltimes_{\Delta, \red} C^\red(G)) \otimes A
\] 
is sequentially split. Moreover, by the Takesaki-Takai duality Theorem \ref{TTduality}, 
we have $G \ltimes_{\Delta, \red} C^\red(G) \cong \IK(L^2(G))$, and the resulting 
map $ G \ltimes_{\alpha, \red} A \rightarrow \IK(L^2(G)) \otimes A $ is the standard embedding. 

The asserted permanence properties are then a consequence of \cite[Theorem 2.9]{BarlakSzabo15}.
\end{proof}

\subsection{Spatial approximate representability} 

Let us now define spatial approximate representability. 

Let $G$ be a discrete quantum group and $ (A,\alpha) $ a separable $G$-\cstar-algebra. 
Denote by $W_\alpha = (\id \otimes \iota_G)(W) \in M(C_0(G)\otimes (G \ltimes_{\alpha,\red} A))$ where 
$\iota_G: C^*_\red(G) \rightarrow M(G \ltimes_{\alpha, \red} A) $ is the canonical embedding. 
The unitary $ W_\alpha $ implements the inner action of $ G $ on the crossed product, more precisely 
$ \ad(W_\alpha^*) $ turns $ M(G \ltimes_{\alpha,\red} A) $ into a $ G$-\cstar-algebra such that 
\[
W_\alpha^*(\eins\otimes a)W_\alpha=\alpha(a)
\] 
for all $a\in A \subset M(G \ltimes_{\alpha,\red} A) $. 

\begin{defi} \label{def:ar}
Let $G$ be a discrete quantum group and $ (A,\alpha) $ a separable $G$-\cstar-algebra. 
We say that $\alpha$ is spatially approximately representable if the natural embedding
\[
j_\alpha: (A,\alpha) \to \big(G \ltimes_{\alpha, \red} A, \ad(W_\alpha^*) \big)
\]
is $ G $-equivariantly sequentially split.
\end{defi}

\begin{prop} \label{char:ar}
Let $G$ be an exact discrete quantum group and $ (A,\alpha) $ a separable $G$-\cstar-algebra. 
\begin{enumerate}[label=\textup{(\alph*)},leftmargin=*] 
\item[a)] If $\alpha$ is spatially approximately representable, then there exists a unitary 
representation $V\in M(C_0(G)\otimes D_{\infty,A})$ of $G$ such that
\[
(\id\otimes\alpha_\infty)(V) = V_{23}^*V_{13}V_{23}
\]
and
\[
V^*(\eins\otimes a)V=\alpha(a)
\]
for all $a\in A$. 
\item[b)] If $G$ is amenable and there exists a unitary representation $V\in M(C_0(G)\otimes D_{\infty,A})$ as in $ a) $, then $\alpha$ 
is spatially approximately representable. 
\end{enumerate}
\end{prop}
\begin{proof}
$ a) $ Assume that $\alpha$ is spatially approximately representable and let 
\[
\psi: G \ltimes_{\alpha, \red} A \rightarrow A_\infty
\] 
be a $G$-equivariant approximate left-inverse for the embedding 
$ A \rightarrow G \ltimes_{\alpha, \red} A $. 
Since this embedding is nondegenerate, the image of $\psi$ is contained in $D_{\infty,A}$, 
and $ \psi: G \ltimes_{\alpha, \red} A \rightarrow D_{\infty, A} $ is a nondegenerate $*$-homomorphism. 
Let us denote the unique strictly continuous extension of $ \psi $ by the same letter, so that 
\[
\psi: M(G\ltimes_{\alpha, \red} A) \to M(D_{\infty,A}). 
\]
We let
\[
V = (\id \otimes \psi)(W_\alpha) \in M(C_0(G) \otimes D_{\infty,A})
\]
be the unitary representation corresponding to the restriction of $ \psi $ to $ C^*_\red(G) $. 
Equivariance of $ \psi: G \ltimes_{\alpha, \red} A \rightarrow D_{\infty,A} $ means 
\[
(\id \otimes \psi)(W_\alpha^* (\eins \otimes x) W_\alpha) = \alpha_\infty \circ \psi(x) 
\] 
for all $ x \in G \ltimes_{\alpha, \red} A $.
In particular, for every $ a \in A $ we obtain 
\[
\alpha(a) = \alpha_\infty \circ \psi(a) = V^* (1 \otimes a) V. 
\]
Moreover, if $ y = (\omega \otimes \id)(W_\alpha) \in M(G \ltimes_{\alpha, \red} A) $ for $ \omega \in \IL(L^2(G))_* $, 
then 
\[
\begin{array}{ccl}
(\omega \otimes \id)(\id \otimes \alpha_\infty)(V) &=& \alpha_\infty \circ \psi(y) \\
&=& (\id \otimes \psi)(W_\alpha^* (\eins \otimes y) W_\alpha) \\
&=& (\omega \otimes \id \otimes \id)(\id \otimes \id \otimes \psi)(W^*_{23} W_{13} W_{23}) \\
&=& (\omega \otimes \id \otimes \id)(V_{23}^* V_{13} V_{23}).
\end{array}
\]
Since this holds for all $ \omega \in \IL(L^2(G))_* $ we conclude $ (\id \otimes \alpha_\infty)(V) = V_{23}^* V_{13} V_{23} $. 

$ b) $ Suppose that $ V \in M(C_0(G) \otimes D_{\infty,A}) $ is a unitary satisfying the conditions in $ a) $. 
Clearly, the canonical map $\iota: A \rightarrow D_{\infty,A} $ is equivariant and nondegenerate, and the formula 
\[
(\id\otimes\iota)\circ \alpha(a) = V^* (\eins \otimes \iota(a)) V \quad\text{for all}~a\in A
\] 
means that $ \iota $ and $ V $ define a covariant pair. 
Hence they combine to a nondegenerate $ * $-homomorphism 
\[
\psi: G \ltimes_{\alpha, \max} A \rightarrow M(D_{\infty,A})
\]
such that $ \psi \circ \iota(a) = a $ for all $ a \in A $. 
Since $ G $ is amenable, we can identify the full crossed product $ G \ltimes_{\alpha, \max} A $ with the reduced crossed 
product $ G \ltimes_{\alpha, \red} A $. To verify that $ \psi: G \ltimes_{\alpha, \red} A \rightarrow M(D_{\infty,A}) $ 
is $ G $-equivariant, it suffices to check this separately on the copies of $ C^*_\red(G) $ and $ A $ inside $ M(G \ltimes_{\alpha, \red} A) $. 

For $ a \in A \subset G \ltimes_{\alpha, \red} A $, the equivariance condition follows immediately from the 
relation $ \psi \circ \iota(a) = a $. On $ C^*_\red(G) $ it is obtained by slicing the equation
\[
\begin{array}{ccl}
(\id \otimes \beta_\infty)(\id \otimes \psi)(W_\beta) &=& (\id \otimes \beta_\infty)(V) \\
&=& V_{23}^* V_{13} V_{23} \\
&=& (\id \otimes \id \otimes \psi)(W_{23}^* W_{13} W_{23}). 
\end{array}
\]
in the first tensor factor and using $ C^*_\red(G) = [(\IL(L^2(G))_* \otimes \id)(W)] $. 
We conclude that $ \psi $ determines a $ G $-equivariant approximate left-inverse for the 
inclusion $ A \rightarrow G \ltimes_{\alpha, \red} A $. 
\end{proof}

\begin{rem} \label{spatialremark2}
Definition \ref{def:ar} generalizes approximate representability for actions of discrete amenable groups, 
see \cite[4.23]{BarlakSzabo15}. In the same way as already indicated in Remark \ref{spatialremark1}, the term \emph{spatial} 
in our definition is included since we work with minimal (braided) tensor products and reduced crossed products. 
In fact, approximate representability for classical discrete groups is defined in terms of the full crossed product instead, 
see \cite[4.23]{BarlakSzabo15}. Notice that the trivial action of the free group $\IF_2$ on $\IC$ is clearly approximately 
representable, but it is easily seen \emph{not} to be spatially approximately representable in the sense of Definition \ref{def:ar}.

It would therefore be more natural to develop the theory with maximal tensor products and full crossed products instead. 
However, this would mean in particular that one would have to work with full coactions taking values in maximal tensor products, which is technically less convenient. 

Let us point out that all the above mentioned issues disappear for coamenable compact quantum groups and amenable discrete quantum groups, respectively; in these cases, we may omit the term \emph{spatial}, and speak of the Rokhlin property and approximate representability.
\end{rem}

\subsection{Duality} 

We shall now show in several steps that the spatial Rokhlin property and spatial approximate representability are dual 
to each other.

\begin{prop} \label{aux iso Psi_alpha}
Let $G$ be a compact quantum group and $ (A, \alpha) $ a separable $G$-\cstar-algebra. Consider the $G$-equivariant $*$-homo\-morphism
\[
\iota_\alpha=\eins\boxtimes\id_A: (A,\alpha) \to \big( C^\red(G)\boxtimes A, \Delta\boxtimes\alpha \big).
\]
Then there exists a $\check{G}$-equivariant $*$-isomorphism
\[
\Psi_\alpha : \big(G \ltimes_{\Delta\boxtimes\alpha, \red} (C^\red(G)\boxtimes A), (\Delta\boxtimes\alpha)\check{\phantom{a}}\big) 
\to \big(\check{G} \ltimes_{\check{\alpha}, \red} ( G \ltimes_{\alpha, \red} A ), \ad(W_{\check{\alpha}}^*)\big)
\]
that makes the following diagram commutative: 
\[
\xymatrix{
(G \ltimes_{\alpha, \red} A, \check{\alpha}) \ar[rr]^{\hspace{-15mm}G\ltimes \iota_\alpha} \ar[rrd]_{j_{\check{\alpha}}} && 
\big( G \ltimes_{\Delta\boxtimes\alpha, \red} (C^\red(G)\boxtimes A), (\Delta\boxtimes\alpha)\check{\phantom{a}} \big) \ar[d]^{\Psi_\alpha} \\
&&  \big( \check{G} \ltimes_{\check{\alpha}, \red} ( G \ltimes_{\alpha, \red} A ), \ad(W_{\check{\alpha}}^*)  \big)
}
\]
\end{prop}
\begin{proof}
Using Lemma \ref{trivialization} and Theorem \ref{TTduality}, we obtain $\Psi_\alpha$ as the composition of the following identifications: 
\[
\begin{array}{ccl}
G \ltimes_{\Delta\boxtimes\alpha, \red} (C^\red(G) \boxtimes A) &\cong& G \ltimes_{\Delta\otimes\id_A, \red} (C^\red(G) \otimes A) \\
&\cong & [(C_\red^*(G) \otimes \eins \otimes \eins) \Delta(C^\red(G))_{12} (\eins \otimes \eins \otimes A)] \\ 
&\cong & [C_\red^*(G) C^\red(G) \otimes A] \\
&=& \IK(L^2(G)) \otimes A \\ 
&=& [(UC^\red(G)U C_\red^*(G) \otimes \eins) \alpha(A)] \\
&\cong & \check{G} \ltimes_{\check{\alpha}, \red} G \ltimes_{\alpha, \red} A.
\end{array}
\]
Note that the copy of $ A $ inside $ M(G \ltimes_{\Delta \boxtimes \alpha, \red} (C^\red(G) \boxtimes A)) $ identifies 
with $ \alpha(A) \subset M(\IK(L^2(G)) \otimes A) $, 
and that the same holds for the copy of $ A $ inside $ M(\check{G} \ltimes_{\check{\alpha}, \red} G \ltimes_{\alpha, \red} A) $. 
Similary, the copies of $ C_\red^*(G) $ on both sides identify with $ C_\red^*(G) \otimes \eins \subset M(\IK(L^2(G)) \otimes A) $. 

Moreover, the above identifications are compatible with the action of $ \check{G} $ on $ \IK(L^2(G)) \otimes A $ implemented by 
conjugation with $ \Sigma \hat{V} \Sigma $. More precisely, the coaction 
\[ 
T \mapsto \ad(\Sigma_{12} \hat{V}_{12} \Sigma_{12})(\eins \otimes T)
\] 
on $\IK(L^2(G)) \otimes A$ corresponds to the dual coaction 
on $ G \ltimes_{\Delta \boxtimes \alpha, \red} (C^\red(G) \boxtimes A) $ 
and to the conjugation coaction $ \gamma = \ad(W_{\check{\alpha}}^*): \check{G} \ltimes_{\hat{\alpha},\red} G \ltimes_{\alpha, \red} A 
\to M(C_\red^*(G)^\cop \otimes \check{G} \ltimes_{\check{\alpha}, \red} G \ltimes_{\alpha, \red} A) $, 
given by $ \gamma(T) = \hat{W}^U_{12}(\eins \otimes T) (\hat{W}^U)^*_{12} $ where 
$ W_{\check{\alpha}}^* = \hat{W}^U = (\eins \otimes U)\hat{W}(\eins \otimes U) $. 
For the latter observe $ \Sigma \hat{V} \Sigma = (\eins \otimes U)\hat{W}(\eins \otimes U) $ and take into account the passage 
from $ C_\red^*(G) $ to $ C_\red^*(G)^\cop $.
\end{proof}

\begin{prop} \label{aux iso Phi_beta}
Let $G$ be a discrete quantum group and $ (A, \alpha) $ a separable $G$-\cstar-algebra. Consider the $G$-equivariant inclusion
\[
j_\alpha: (A,\alpha) \to \big(G\ltimes_{\alpha,\red} A, \ad(W_\alpha^*) \big).
\]
Then there exists a $\check{G}$-equivariant $*$-isomorphism
\[
\Phi_\alpha: \big(G \ltimes_{\ad(W_\alpha^*),\red} (G \ltimes_{\alpha, \red} A), \ad(W_\alpha^*)\check{\phantom{a}} \big) 
\to \big(C^\red(\check{G})\boxtimes (G \ltimes_{\alpha,\red} A), \check{\Delta}\boxtimes\check{\alpha} \big)
\]
that makes the following diagram commutative: 
\[
\xymatrix{
(G \ltimes_{\alpha,\red} A, \check{\alpha}) \ar[rr]^{\hspace{-15mm} G \ltimes_\red j_\alpha} \ar[rrd]_{\iota_{\check{\alpha}}} 
&& \big(G\ltimes_{\ad(W_\alpha^*), \red} (G \ltimes_{\alpha, \red} A) , \ad(W_\alpha^*)\check{\phantom{a}} \big) \ar[d]^{\Phi_\alpha}  \\
&& \big( C^\red(\check{G})\boxtimes (G \ltimes_{\alpha, \red} A) , \check{\Delta}\boxtimes\check{\alpha} \big)
}
\]
\end{prop}
\begin{proof}
We obtain $\Phi_\alpha$ as the composition of the following identifications: 
\[
\begin{array}{ccl}
G \ltimes_{\ad(W_\alpha^*)} (G \ltimes_{\alpha, \red} A) &=&
[(C_\red^*(G) \otimes \eins \otimes \eins) W^*_{12} (\eins \otimes C_\red^*(G) \otimes \eins) \alpha(A)_{23} W_{12}] \\
&=& [W_{12}^* (\hat{\Delta}^\cop(C_\red^*(G)) \otimes \eins) (\eins \otimes C_\red^*(G) \otimes \eins) \alpha(A)_{23} W_{12}] \\ 
&\cong& [(\hat{\Delta}^\cop(C_\red^*(G)) \otimes \eins) (\eins \otimes C_\red^*(G) \otimes \eins) \alpha(A)_{23} ] \\ 
&=& [(C_\red^*(G) \otimes C_\red^*(G) \otimes \eins) \alpha(A)_{23}] \\ 
&\cong& C^*_\red(G)^{\cop} \otimes (G \ltimes_{\alpha, \red} A) \\ 
&=& C^\red(\check{G}) \otimes (G \ltimes_{\alpha, \red} A) \\ 
&\cong& C^\red(\check{G}) \boxtimes (G \ltimes_{\alpha, \red} A).  
\end{array}
\]
Under these identifications, the copy of $ C_\red^*(G) $ inside $ M(G \ltimes_{\alpha, \red} A) $ on the left hand side gets 
identified with $\eins \boxtimes (C^*_\red(G) \otimes \eins) $ 
inside $ C^\red(\check{G}) \boxtimes (G \ltimes_{\alpha, \red} A) $, and the copy of $ A $ in $ G \ltimes_{\alpha, \red} A $
is mapped to $ \eins \boxtimes \alpha(A) $. 
In other words, we indeed obtain a commutative diagram as desired.

Moreover, it is not hard to check that the dual action on $ G \ltimes_{\ad(W_\alpha^*), \red} (G \ltimes_{\alpha, \red} A) $ 
corresponds to the action of $ \hat{\Delta}^\cop = \check{\Delta} $ 
on the first tensor factor of $ C^\red(\check{G}) \otimes (G \ltimes_{\alpha, \red} A) $. 
It follows that $\Phi_\alpha$ is $ \check{G} $-equivariant.
\end{proof}

As a consequence, we obtain the duality between the spatial Rokhlin property and spatial approximate representability. 

\begin{theorem} \label{duality rp ar}
Let $G$ be a coexact compact quantum group and let $(A, \alpha)$ be a separable $G$-\cstar-algebra. 
Then $ \alpha $ has the spatial Rokhlin property if and only if $\check{\alpha}$ is spatially approximately representable. 

Dually, let $ G $ be an exact discrete quantum group and let $ (A, \alpha) $ be a separable $G$-\cstar-algebra. 
Then $ \alpha $ is spatially approximately representable if and only if $\check{\alpha}$ has the spatial Rokhlin property. 
\end{theorem}
\begin{proof}
Let us first consider the case of compact quantum groups. By the general duality result from 
Proposition \ref{general-duality}, we know that
\[
\iota_\alpha: (A,\alpha)\into \big(C^\red(G)\boxtimes A,\Delta \boxtimes \alpha \big)
\]
is $G$-equivariantly sequentially split if and only if the induced $*$-homomorphism 
\[
G\ltimes_\red \iota_\alpha: (G \ltimes_{\alpha,\red} A, \check{\alpha}) \to 
\big(G \ltimes_{\Delta\boxtimes\alpha, \red} (C(G)\boxtimes A), (\Delta\boxtimes\alpha)\check{\phantom{a}} \big) 
\]
is $\check{G}$-equivariantly sequentially split. By Proposition \ref{aux iso Psi_alpha}, there exists a commutative diagram 
of $\check{G}$-equivariant $*$-homomorphisms
\[
\xymatrix{
(G \ltimes_{\alpha, \red} A, \check{\alpha}) \ar[rr]^{\hspace{-15mm}G\ltimes_\red \iota_\alpha} \ar[rrd]_{j_{\check{\alpha}}} && 
\big(G \ltimes_{\Delta\boxtimes\alpha, \red} (C(G) \boxtimes A), (\Delta\boxtimes\alpha)\check{\phantom{a}} \big) \ar[d]^\cong \\
&&  \big(\check{G} \ltimes_{\check{\alpha}, \red} (G \ltimes_{\alpha, \red} A ), \ad(W_{\check{\alpha}}^*)  \big)
}
\]
We conclude that $\iota_\alpha$ is $G$-equivariantly sequentially split if and only if $j_{\check{\alpha}}$ is $\check{G}$-equivariantly sequentially split. This means that $\alpha$ has the spatial Rokhlin property if and only if $\check{\alpha}$ is spatially 
approximately representable. 

The claim in the discrete case is proved in an analogous fashion. Again by Proposition \ref{general-duality}, 
the $G$-equivariant $*$-homomorphism
\[
j_\alpha: (A,\alpha) \to \big(G \ltimes_{\alpha, \red} A, \ad(W_\alpha^*) \big)
\]
is $G$-equivariantly sequentially split if and only if the induced $*$-homomorphism
\[
G \ltimes_\red j_\alpha:(G \ltimes_{\alpha, \red} A, \check{\alpha}) \to 
\big(G \ltimes_{\ad(W_\alpha^*), \red} (G \ltimes_{\alpha, \red} A), \ad(W_\alpha^*)\check{\phantom{a}} 
\big)
\]
is $\check{G}$-equivariantly sequentially split. By Proposition \ref{aux iso Phi_beta}, there exists a commutative diagram 
of $\check{G}$-equivariant $*$-homomorphisms
\[
\xymatrix{
(G \ltimes_{\alpha, \red} A, \check{\alpha}) \ar[rr]^{\hspace{-15mm} G \ltimes_\red j_\alpha} \ar[rrd]_{\iota_{\check{\alpha}}} && 
\big(G \ltimes_{\ad(W_\alpha^*), \red} (G \ltimes_{\alpha, \red} A) , \ad(W_\alpha^*)\check{\phantom{a}} \big) \ar[d]^\cong  \\
&& \big( C^\red(\check{G})\boxtimes (G \ltimes_{\alpha, \red} A) , \check{\Delta}\boxtimes\check{\alpha} \big)
}
\]
We conclude that $j_\alpha$ is $G$-equivariantly sequentially split if and only if $\iota_{\check{\alpha}}$ is $\check{G}$-equivariantly sequentially split. Hence $\alpha$ is spatially approximately representable if and only if $\check{\alpha}$ has the spatial Rokhlin property. 
\end{proof}


\section{Rigidity of Rokhlin actions}

In this section we provide a classification of actions of coexact compact quantum groups with the spatial Rokhlin property on 
separable \cstar-algebras. This type of result was first obtained by Izumi in \cite{Izumi04}. Our basic approach follows 
Gardella-Santiago \cite[Section 3]{GardellaSantiago15}, who proved corresponding results for finite group actions. We note 
that Gardella-Santiago have also announced the results of this section for actions of classical compact groups, see \cite{GardellaSantiago16}.

Recall that if $ (B, \beta) $ is a $G$-\cstar-algebra for a compact quantum group $ G $ 
then $ B^\beta \subset B $ denotes the fixed point subalgebra. We shall also write $ \beta $ for the induced coaction 
on the minimal unitarization $ \tilde{B} $ of $B$; note that $ \beta(\eins) = \eins \otimes \eins $. 
\begin{defi}
Let $G$ be a compact quantum group. Let $\alpha: A\to C^\red(G)\otimes A$ and $\beta: B\to C^\red(G)\otimes B$ be two $G$-actions 
on \cstar-algebras, and assume that $A$ is separable. Let $\phi_1, \phi_2: (A,\alpha)\to (B,\beta)$ be two equivariant $*$-homomorphisms. We say that $\phi_1$ and $\phi_2$ are approximately $G$-unitarily equivalent, written $\phi_1\ue{G}\phi_2$, if there exists a sequence of unitaries $v_n\in\CU(\tilde{B}^\beta)$ such that
\[
\phi_2(x)=\lim_{n\to\infty} v_n\phi_1(x)v_n^* \quad\text{for all}~x\in A.
\]
\end{defi}

\begin{rem}
For the trivial (quantum) group $G$, the above definition recovers the usual notion of approximate unitary equivalence between $*$-homomorphisms. We write simply $\phi_1\ueo \phi_2$ instead of $\phi_1\ue{G}\phi_2$ in this case. 
\end{rem}

\begin{prop} \label{intertwining}
Let $G$ be a compact quantum group. Let $\alpha: A\to C^\red(G)\otimes A$ and $\beta: B\to C^\red(G)\otimes B$ be two $G$-actions on separable \cstar-algebras. Let $\phi: (A,\alpha)\to (B,\beta)$ and $\psi: (B,\beta)\to (A,\alpha)$ be two equivariant $*$-homomorphisms such that $\psi\circ\phi\ue{G}\id_A$ and $\phi\circ\psi\ue{G}\id_B$. Then there exists an equivariant $ * $-isomorphism $\Phi: (A,\alpha)\to (B,\beta)$ such that $\Phi\ue{G}\phi$ and $\Phi^{-1}\ue{G}\psi$.
\end{prop}
\begin{proof}
This follows from a straightforward adaptation of the proof of \cite[Corollary~2.3.4]{Rordam2001} to the setting of $G$-equivariant $*$-homomorphism. For this, one requires the approximate intertwinings from \cite[Definition~2.3.1]{Rordam2001} to be (approximately) $G$-equivariant in the obvious way. The resulting $*$-isomorphism $\Phi: A\to B$ then turns out to be $\alpha$-to-$\beta$-equivariant. 
Moreover, the approximate unitary equivalences $\Phi\ueo \phi$ and $\Phi^{-1}\ueo\psi$ that come out of the proof are indeed implemented by unitaries in $\CU(\tilde{B}^\beta)$ and $\CU(\tilde{A}^\alpha)$, respectively.
\end{proof}

Let us now consider a series of partial results that will lead to the classification of Rokhlin actions.

\begin{lemma} \label{pre-existence-1}
Let $G$ be a compact quantum group. Let $\alpha: A\to C^\red(G)\otimes A$ and $\beta: B\to C^\red(G)\otimes B$ be two $G$-actions on separable \cstar-algebras. Let $\phi: A\to B$ be a $*$-homomorphism that is equivariant modulo approximate unitary equivalence, i.e.\ $\beta\circ\phi \ueo (\id\otimes\phi)\circ\alpha$ as $*$-homomorphisms between $A$ and $C^\red(G)\otimes B$. Then for every finite set $F\fin A$ and 
every $\eps>0$, there exists a unitary $v\in (C^\red(G)\boxtimes B)^\sim$ such that
\[
(\Delta\boxtimes\beta)\circ \ad(v)\circ(\eins\boxtimes\phi)(x) =_\eps \big( \id\otimes (\ad(v)\circ(\eins\boxtimes\phi)) \big)\circ\alpha(x) 
\]
and
\[
\| [(\eins\boxtimes \phi)(x), v] \| \leq \eps+\|\beta\circ\phi(x)-(\id\otimes\phi)\circ\alpha(x)\|
\]
for all $x\in F$.
\end{lemma}
\begin{proof}
For convenience, the term $\id$ will always denote the identity map on $C^\red(G)$ in this proof. Identity maps on other sets are decorated with the corresponding set.

Using our assumptions on $\phi$, we may choose unitaries $u_n\in (C^\red(G)\otimes B)^\sim$ such that
\begin{equation} \label{e:pre-ex:1}
\ad(u_n)\circ\beta\circ\phi \quad\stackrel{n\to\infty}{\longrightarrow}\quad (\id\otimes\phi)\circ\alpha
\end{equation}
in point-norm.
By Lemma \ref{trivialization}, we have the equivariant isomorphism
\begin{equation} \label{e:pre-ex:2}
T_\beta: ( C^\red(G)\boxtimes B, \Delta\boxtimes\beta) \to ( C^\red(G)\otimes B, \Delta\otimes\id_B )
\end{equation}
satisfying 
\begin{equation} \label{e:pre-ex:3}
T_\beta(\eins\boxtimes x)=\beta(x)\quad\text{for all}~x\in B.
\end{equation}
Let us also denote by $T_\beta$ the obvious extension to the unitarizations. Set $v_n=T_\beta^{-1}(u_n)$. We calculate 
\[
\begin{array}{cl}
\multicolumn{2}{l}{ \hspace{-5mm} \dst \lim_{n\to\infty}~ (\Delta\boxtimes\beta)\circ \ad(v_n)\circ(\eins\boxtimes\phi) } \\
=& \dst \lim_{n\to\infty}~ (\Delta\boxtimes\beta)\circ \ad(T_\beta^{-1}(u_n))\circ(\eins\boxtimes\phi) \\
\stackrel{\eqref{e:pre-ex:3}}{=}& \dst \lim_{n\to\infty}~ (\Delta\boxtimes\beta)\circ T_\beta^{-1}\circ\ad(u_n)\circ\beta\circ\phi \\ 
\stackrel{\eqref{e:pre-ex:1}}{=} & (\Delta\boxtimes\beta)\circ T_\beta^{-1}\circ(\id\otimes\phi)\circ\alpha \\
\stackrel{\eqref{e:pre-ex:2}}{=} & (\id\otimes T_\beta^{-1})\circ (\Delta\otimes\id_B)\circ (\id\otimes\phi)\circ\alpha \\
=& (\id\otimes T_\beta^{-1})\circ (\id\otimes\id\otimes\phi)\circ(\Delta\otimes\id_A)\circ\alpha \\
=& (\id\otimes T_\beta^{-1})\circ (\id\otimes\id\otimes\phi)\circ (\id\otimes\alpha)\circ\alpha \\
=& \big( \id\otimes (T_\beta^{-1}\circ (\id\otimes\phi)\circ\alpha) \big)\circ\alpha \\
\stackrel{\eqref{e:pre-ex:1}}{=} & \dst \lim_{n\to\infty}~ \big( \id\otimes (T_\beta^{-1}\circ \ad(u_n)\circ\beta\circ\phi ) \big)\circ\alpha \\
=& \dst \lim_{n\to\infty}~ \big( \id\otimes( \ad(v_n)\circ(\eins\boxtimes\phi)) \big)\circ\alpha.
\end{array}
\]
One should note that even though the existence of all these limits is a priori not clear at the beginning of this calculation, it follows a posteriori from the steps in this calculation.

Moreover, we calculate for all $x\in A$ that
\[
\begin{array}{cl}
\multicolumn{2}{l}{ \hspace{-5mm} \|[\eins\boxtimes \phi(x), v_n]\| } \\
=& \|[(\eins\boxtimes\phi)(x), T_\beta^{-1}(u_n)] \| \\
\stackrel{\eqref{e:pre-ex:3}}{=}& \|[(\beta\circ\phi)(x),u_n]\| \\
=& \|(\ad(u_n)\circ\beta\circ\phi)(x)-(\beta\circ\phi)(x)\| \\
\stackrel{\eqref{e:pre-ex:1}}{\longrightarrow} & \|(\id\otimes\phi)\circ\alpha(x)-\beta\circ\phi(x)\|.
\end{array}
\]
From these two calculations, it is clear that for given $F\fin A$ and $\eps>0$, any of the unitaries $v_n$ satisfies the desired property for sufficiently large $n$.
\end{proof}

\begin{lemma} \label{pre-existence-2}
Let $G$ be a coexact compact quantum group. Let $\alpha: A\to C^\red(G)\otimes A$ and $\beta: B\to C^\red(G)\otimes B$ be two $G$-actions on separable \cstar-algebras. Assume that $\beta$ has the spatial Rokhlin property. Let $\phi: A\to B$ be a $*$-homomorphism that is equivariant modulo approximate unitary equivalence, i.e.\ $\beta\circ\phi \ueo (\id\otimes\phi)\circ\alpha$ as $*$-homomorphisms between $A$ and $C^\red(G)\otimes B$. Then for every finite set $F\fin A$ and $\eps>0$, there exists a unitary $v\in \tilde{B}$ such that
\[
\beta\circ \ad(v)\circ\phi(x) =_\eps \big( \id\otimes (\ad(v)\circ\phi) \big)\circ\alpha(x) 
\]
and
\[
\| [\phi(x), v] \| \leq \eps+\|\beta\circ\phi(x)-(\eins\otimes\phi)\circ\alpha(x)\|
\]
for all $x\in F$.
\end{lemma}
\begin{proof}
As $\beta$ is assumed to have the spatial Rokhlin property, let 
\[
\psi: (C^\red(G)\boxtimes B,\Delta\boxtimes\beta)\to (B_\infty,\beta_\infty)
\] 
be an equivariant $*$-homomorphism satisfying
\begin{equation} \label{e:pre-ex:4}
\psi(\eins\boxtimes b)=b\quad\text{for all } b \in B.
\end{equation}
We also denote by $\psi$ the canonical extensions to the smallest unitarizations on both sides. Now let $F\fin A$ and $\eps>0$ be given. 
Apply Lemma \ref{pre-existence-1} and choose a unitary $w\in (C^\red(G)\boxtimes B)^\sim$ such that
\begin{equation} \label{e:pre-ex:5}
(\Delta\boxtimes\beta)\circ \ad(w)\circ(\eins\boxtimes\phi)(x) =_{\eps/2} \big( \id\otimes (\ad(w)\circ(\eins\boxtimes\phi)) \big)\circ\alpha(x) 
\end{equation}
and
\begin{equation} \label{e:pre-ex:6}
\| [(\eins\boxtimes \phi)(x), w] \| \leq \eps/2+\|\beta\circ\phi(x)-(\id\otimes\phi)\circ\alpha(x)\|
\end{equation}
for all $x\in F$. Set $v=\psi(w)\in (B_\infty)^\sim \subset \tilde{B}_\infty$.
Combining the equivariance of $\psi$ with \eqref{e:pre-ex:4}, \eqref{e:pre-ex:5} and \eqref{e:pre-ex:6}, we obtain 
\[
\beta_\infty \circ \ad(v)\circ\phi(x) =_{\eps/2} \big( \id\otimes (\ad(v)\circ\phi) \big)\circ\alpha(x)
\]
and
\[
\| [\phi(x), v] \| \leq \eps/2+\|\beta\circ\phi(x)-(\id\otimes\phi)\circ\alpha(x)\|
\]
for all $x\in F$. Now represent $v$ by some sequence of unitaries $v_n\in\tilde{B}$. Then these equations translate to the conditions
\[
\limsup_{n\to\infty}~ \|\beta \circ \ad(v_n)\circ\phi(x) - \big( \id\otimes (\ad(v_n)\circ\phi) \big)\circ\alpha(x)\| \leq \eps/2
\]
and
\[
\limsup_{n\to\infty}~ \| [\phi(x), v_n] \| \leq \eps/2+\|\beta\circ\phi(x)-(\id\otimes\phi)\circ\alpha(x)\|
\]
for all $x\in F$. It follows that for sufficiently large $n$, any of the unitaries $v_n$ satisfies the desired inequalities with respect to $\eps$ in place of $\eps/2$. This finishes the proof.
\end{proof}

\begin{prop}[cf.\ {\cite[3.2]{GardellaSantiago15}}] \label{existence}
Let $G$ be a coexact compact quantum group. Let $\alpha: A\to C^\red(G)\otimes A$ and $\beta: B\to C^\red(G)\otimes B$ be two $G$-actions on separable \cstar-algebras. Assume that $\beta$ has the spatial Rokhlin property. Let $\phi: A\to B$ be a $*$-homomorphism that is equivariant modulo approximate unitary equivalence, i.e.\ $\beta\circ\phi \ueo (\id\otimes\phi)\circ\alpha$ as $*$-homomorphisms between $A$ and $C^\red(G)\otimes B$. Then there exists an equivariant $*$-homomorphism $\psi: (A,\alpha)\to (B,\beta)$ with $\psi\ueo\phi$.
\end{prop}
\begin{proof}
Let 
\[
F_1\subset F_2\subset F_3\subset \dots \fin A
\]
be an increasing sequence of finite subsets with dense union. Let $ (\eps_n)_{n \in \mathbb{N}} $ be a decreasing sequence 
of strictly positive numbers with $\sum_{n=1}^\infty \eps_n < \infty$. Using Lemma \ref{pre-existence-2} we find a 
unitary $v_1\in\tilde{B}$ satisfying
\[
\beta\circ \ad(v_1)\circ\phi(x) =_{\eps_1} \big( \id\otimes (\ad(v_1)\circ\phi) \big)\circ\alpha(x) 
\]
for all $x\in F_1$. Applying Lemma \ref{pre-existence-2} again (but now for $\ad(v_1)\circ\phi$ in place of $\phi$), we find a unitary $v_2\in\tilde{B}$ satisfying
\[
\beta\circ \ad(v_2v_1)\circ\phi(x) =_{\eps_2} \big( \id\otimes (\ad(v_2v_1)\circ\phi) \big)\circ\alpha(x) 
\]
and
\[
\| [(\ad(v_1)\circ\phi)(x), v_2] \| \leq \eps_2+\|\beta\circ\ad(v_1)\circ\phi(x)-(\id\otimes(\ad(v_1)\circ\phi))\circ\alpha(x)\|
\]
for all $x\in F_2$. Applying Lemma \ref{pre-existence-2} again (but now for $\ad(v_2v_1)\circ\phi$ in place of $\phi$), we 
find a unitary $v_3\in\tilde{B}$ satisfying
\[
\beta\circ \ad(v_3v_2v_1)\circ\phi(x) =_{\eps_3} \big( \id\otimes (\ad(v_3v_2v_1)\circ\phi) \big)\circ\alpha(x) 
\]
and
\[
\| [(\ad(v_2v_1)\circ\phi)(x), v_3] \| \leq \eps_3+\|\beta\circ\ad(v_2v_1)\circ\phi(x)-(\eins\otimes(\ad(v_2v_1)\circ\phi))\circ\alpha(x)\|
\]
for all $x\in F_3$. We inductively repeat this process and obtain a sequence of unitaries $v_n\in\tilde{B}$ satisfying
\begin{equation} \label{e:ex:1}
\beta\circ \ad(v_n\cdots v_1)\circ\phi(x) =_{\eps_n} \big( \id\otimes (\ad(v_n\cdots v_1)\circ\phi) \big)\circ\alpha(x) 
\end{equation}
for all $n\geq 1$ and
\begin{equation} \label{e:ex:2}
\begin{array}{rl}
\multicolumn{2}{l}{ \| [(\ad(v_{n-1}\cdots v_1)\circ\phi)(x), v_n] \| } \\ \leq & \eps_n+\|\beta\circ\ad(v_{n-1}\cdots v_1)\circ\phi(x)-(\id\otimes(\ad(v_{n-1}\cdots v_n)\circ\phi))\circ\alpha(x)\|
\end{array}
\end{equation}
for all $x\in F_n$ and $n\geq 2$. 
For $m>n\geq k$ and $x\in F_k$ this implies
\[
\begin{array}{cl}
\multicolumn{2}{l}{ \hspace{-5mm} \|\ad(v_m\cdots v_1)\circ\phi(x)-\ad(v_n\cdots v_1)\circ\phi(x)\| } \\
\leq & \dst \sum_{j=n}^{m-1} \|\ad(v_{j+1}\cdots v_1)\circ\phi(x)-\ad(v_j\cdots v_1)\circ\phi(x)\| \\
=& \dst \sum_{j=n}^{m-1} \big\| \big[ (\ad(v_j\cdots v_1)\circ\phi)(x), v_{j+1} \big] \big\| \\
\stackrel{\eqref{e:ex:1},\eqref{e:ex:2}}{\leq} & \dst \sum_{j=n}^{m-1}~ \eps_{j+1}+\eps_j ~\leq~ 2\cdot\sum_{j=n}^{m} \eps_j. 
\end{array}
\]
As the $\eps_n$ were chosen as a $1$-summable sequence and the union of the $F_n$ is dense, this estimate implies 
that the sequence $\ad(v_n\cdots v_1)\circ\phi(x)$ is Cauchy for every $x\in A$. In particular, the point-norm limit $\psi=\lim_{n\to\infty} \ad(v_n\cdots v_1)\circ\phi$ exists and yields a well-defined $*$-homomorphism from $A$ to $B$. By construction we have $\psi\ueo\phi$, and the equivariance condition
\[
\beta\circ\psi = (\id\otimes\psi)\circ\alpha
\]
follows from \eqref{e:ex:1}. This finishes the proof.
\end{proof}

\begin{lemma} \label{hom-rigid-1}
Let $G$ be a compact quantum group. Let $\alpha: A\to C^\red(G)\otimes A$ and $\beta: B\to C^\red(G)\otimes B$ be two $G$-actions on separable \cstar-algebras. Let $\phi_1,\phi_2: (A,\alpha)\to (B,\beta)$ be two equivariant $*$-homomorphisms. If $\phi_1\ueo\phi_2$, then $\eins\boxtimes\phi_1\ue{G}\eins\boxtimes\phi_2$ as equivariant $*$-homomorphisms from $A$ to $C^\red(G)\boxtimes B$.
\end{lemma}
\begin{proof}
Let $u_n\in\CU(\tilde{B})$ be a sequence of unitaries satisfying
\begin{equation} \label{e:pre-rigid}
\ad(u_n)\circ\phi_1 \stackrel{n\to\infty}{\longrightarrow} \phi_2.
\end{equation}
Using Lemma \ref{trivialization}, we consider the equivariant isomorphism
\[
T_\beta: ( C^\red(G)\boxtimes B, \Delta\boxtimes\beta) \to ( C^\red(G)\otimes B, \Delta\otimes\id_B )
\]
that satisfies condition \eqref{e:pre-ex:3}.
We shall also denote by $T_\beta$ the obvious extension to the unitarizations. Set 
\[
v_n=T_\beta^{-1}(\eins\otimes u_n)~ \in \big( C^\red(G)\boxtimes B \big)^\sim.
\] 
As $\eins\otimes u_n$ is in the fixed-point algebra of $\Delta\otimes\id_B$, it follows that $v_n$ is in the fixed-point algebra of $\Delta\boxtimes\beta$. We have
\[
\begin{array}{ccl}
\ad(v_n)\circ(\eins\boxtimes\phi_1) &\stackrel{\eqref{e:pre-ex:3}}{=}& T_\beta^{-1}\circ\ad(\eins\otimes u_n)\circ\beta\circ\phi_1 \\
&=& T_\beta^{-1}\circ\ad(\eins\otimes u_n)\circ (\eins\otimes\phi_1)\circ\alpha \\
&\stackrel{\eqref{e:pre-rigid}}{\longrightarrow}& T_\beta^{-1}\circ(\eins\otimes\phi_2)\circ\alpha \\
&=& T_\beta^{-1}\circ\beta\circ\phi_2 \\
&\stackrel{\eqref{e:pre-ex:3}}{=}& \eins\boxtimes\phi_2.
\end{array}
\]
This shows our claim.
\end{proof}

\begin{prop} \label{ueG pushed forward seq split}
Let $G$ be a coexact compact quantum group. Let $\alpha: A\to C^\red(G)\otimes A$, $\beta: B\to C^\red(G)\otimes B$ and $\gamma: C\to C^\red(G)\otimes C$ be three $G$-actions on separable \cstar-algebras. Let
\[
\phi_1, \phi_2: (A,\alpha)\to (B,\beta),\quad \psi: (B,\beta)\to (C,\gamma)
\]
be equivariant $*$-homomorphisms. Assume that $\psi$ is equivariantly sequentially split. Then $\phi_1\ue{G}\phi_2$ if and only if $\psi\circ\phi_1\ue{G}\psi\circ\phi_2$.
\end{prop}
\begin{proof}
If $\phi_1\ue{G}\phi_2$, then clearly $\psi\circ\phi_1\ue{G}\psi\circ\phi_2$. For the converse, assume that $\psi\circ\phi_1\ue{G}\psi\circ\phi_2$. Let $\kappa:(C,\gamma) \to (B_\infty,\beta_\infty)$ be an equivariant approximate left-inverse for $\psi$. Then $\kappa \circ \psi \circ \phi_1 \ue{G} \kappa \circ \psi \circ \phi_1$, or in other words, $\phi_1 \ue{G} \phi_2$ as equivariant $*$-homomorphisms from $A$ 
to $B_{\infty, \beta} $. Given $F \fin A$ and $\eps >0$, we therefore find some $u \in \CU((\widetilde{B_{\infty, \beta}})^{\beta_\infty})$ such that 
\[
\ad(u) \circ \phi_1(a)=_\eps \phi_2(a)\ \text{for all}\ a \in F.
\]
According to Lemma \ref{fixed-point-alg}, we may choose a representing 
sequence $(u_n)_{n \in \mathbb{N}} \subset \CU(\tilde{B}^\beta)$ for $u$. Picking a suitable member of this sequence, we find a 
unitary $v \in \CU(\tilde{B}^\beta)$ such that
\[
\ad(v) \circ \phi_1(a)=_{2\eps} \phi_2(a)\ \text{for all}\ a \in F.
\]
This shows that $\phi_1\ue{G}\phi_2$ as equivariant $*$-homomorphisms from $(A, \alpha)$ to $(B, \beta)$.
\end{proof}

\begin{cor}[cf.\ {\cite[3.1]{GardellaSantiago15}}] \label{hom-rigid-2}
Let $G$ be a coexact compact quantum group. Let $\alpha: A\to C^\red(G)\otimes A$ and $\beta: B\to C^\red(G)\otimes B$ be two $G$-actions on separable \cstar-algebras. Assume that $\beta$ has the spatial Rokhlin property. Let $\phi_1, \phi_2: (A,\alpha)\to (B,\beta)$ be two equivariant $*$-homomorphisms. Then $\phi_1\ueo\phi_2$ if and only if $\phi_1\ue{G}\phi_2$.
\end{cor}
\begin{proof}
We have to show that $\phi_1\ueo\phi_2$ implies $\phi_1\ue{G}\phi_2$. By Lemma \ref{hom-rigid-1}, the amplified $*$-homomorphisms are $G$-approximately unitarily equivalent, that is, $\eins \boxtimes \phi_1 \ue{G} \eins \boxtimes \phi_2$. As $\beta$ has the Rokhlin property, the canonical embedding $\eins \boxtimes \id_B:(B,\beta) \to (C^\red(G) \boxtimes B, \Delta \boxtimes \beta)$ is equivariantly sequentially split. Writing $\eins \boxtimes \phi_i = (\eins \boxtimes \id_B) \circ \phi_i$ for $i=1,2$, an application of 
Lemma \ref{ueG pushed forward seq split} yields $\phi_1\ue{G}\phi_2$. This finishes the proof.
\end{proof}

Here comes the main result of this section, which generalizes analogous results for finite group actions due to Izumi \cite[3.5]{Izumi04}, Nawata \cite[3.5]{Nawata16} and Gardella-Santiago \cite[3.4]{GardellaSantiago15}. It also generalizes the corresponding results 
for finite quantum groups by Osaka-Teruya \cite[10.7]{KodakaTeruya15} and for classical compact groups by 
Gardella-Santiago \cite{GardellaSantiago16}.

\begin{theorem} \label{rigidity Rokhlin}
Let $G$ be a coexact compact quantum group. Let $\alpha, \beta: A\to C^\red(G)\otimes A$ be two $G$-actions on a separable \cstar-algebra. Assume that both have the spatial Rokhlin property. Then $\alpha\ueo\beta$ as $*$-homomorphisms if and only if there exists an equivariant isomorphism $\theta: (A,\alpha)\to (A,\beta)$ which is approximately inner as a $*$-automorphism of $A$.
\end{theorem}
\begin{proof}
First assume that $\theta: (A,\alpha)\to (A,\beta)$ is an equivariant $ * $-isomorphism which is approximately inner as a $*$-automorphism. 
Then
\[
\beta\ueo\beta\circ\theta = (\id\otimes\theta)\circ\alpha \ueo \alpha.
\]
Now assume that $\alpha$ and $\beta$ are approximately unitarily equivalent. Then clearly
\[
\beta\circ\id_A=\beta\ueo\alpha=(\id\otimes\id_A)\circ\alpha,
\]
and analogously $\alpha\circ\id_A\ueo(\id\otimes\id_A)\circ\beta$. Since both $\alpha$ and $\beta$ have the spatial Rokhlin property, it follows from Proposition \ref{existence} that there exist equivariant $*$-homomorphisms $\phi_1: (A,\alpha)\to (A,\beta)$ and $\phi_2: (A,\beta)\to (A,\alpha)$ that are both approximately inner as $*$-homomorphisms. Hence Corollary \ref{hom-rigid-2} implies $\phi_1\circ\phi_2\ue{G}\id_A$ and $\phi_2\circ\phi_1\ue{G}\id_A$. According to Proposition \ref{intertwining} we conclude that there exists an equivariant 
$ * $-isomorphism $\theta: (A,\alpha)\to (A,\beta)$ with $\theta\ue{G}\phi_1$. In particular, $\theta$ is also approximately inner as 
a $*$-automorphism.
\end{proof}

To conclude this section, we generalize the $K$-theory formula for fixed-point algebras of Rokhlin actions, which is originally due to Izumi and was recently extended by the first two authors. 

\begin{theorem}[cf.\ {\cite[3.13]{Izumi04} and \cite[4.9]{BarlakSzabo15}}]
Let $G$ be a coexact compact quantum group. Let $\alpha: A\to C^\red(G)\otimes A$ be an action on a separable \cstar-algebra with the spatial Rokhlin property. Then the inclusion $A^\alpha\into A$ is injective in $K$-theory, and its image coincides with the subgroup
\[
K_*(A^\alpha)\cong\set{ x\in K_*(A) \mid K_*(\alpha)(x)=K_*(\eins\otimes\id_A)(x) }.
\]
\end{theorem}
\begin{proof}
If $x \in \im(K_*(A^\alpha) \rightarrow K_*(A)) $, then clearly $ K_*(\alpha)(x) = K_*(\eins \otimes \id_A)(x) $.
For the converse, let $ x = [p] - [\eins_k] $ be an element of $ K_0(A) $, where $ p \in M_n(\tilde{A}) $ and 
$ \eins_k \in M_k(\tilde{A}) \subset M_n(\tilde{A}) $ for some $ k \leq n $ such that $ p - \eins_k \in M_n(A) $. 
Let us write 
\[ 
M_n(\tilde{\alpha}): M_n(\tilde{A}) \rightarrow M_n((C^\red(G) \otimes A)^\sim)
\]
for the canonical extension of $ \alpha $ to unitarizations and matrix amplification. 

Similarly, we write $ M_n((\eins \otimes \id_A)^\sim) $ for the extension of $ \eins \otimes \id_A $.
If $ x $ satisfies $ K_0(\alpha)(x) = K_0(\eins \otimes \id_A)(x) $ then 
\[
[M_n(\tilde{\alpha})(p)] - [M_n(\tilde{\alpha})(\eins_k)] = [M_n((\eins \otimes \id_A)^\sim)(p)] 
- [M_n((\eins \otimes \id_A)^\sim)(\eins_k)] 
\] 
in $ K_0(C^\red(G) \otimes A) $. Notice that $ M_n(\tilde{\alpha})(\eins_k) = M_n((\eins \otimes \id_A)^\sim)(\eins_k) $ by definition 
of $ \tilde{\alpha} $, so that we get 
\[
[M_n(\tilde{\alpha})(p)] = [M_n((\eins \otimes \id_A)^\sim)(p)]. 
\] 
By definition of $ K_0 $, we therefore find natural numbers $ m, l $ such that 
\[
M_n(\tilde{\alpha})(p) \oplus \eins_m \oplus 0_l \sim_{\mathrm{MvN}} M_n((\eins \otimes \id_A)^\sim)(p) \oplus \eins_m \oplus 0_l
\]
in $ M_{n + m + l}((C^\red(G) \otimes A)^\sim) $. Using the equivariant isomorphism 
\[ 
T=T_\alpha^{-1} : \big( C^\red(G) \otimes A, \Delta\otimes\id_A \big) \to \big( C^\red(G) \boxtimes A, \Delta\boxtimes\alpha \big)
\] 
from Lemma \ref{trivialization}, we can view this as a relation in $ M_{n + m + l}((C^\red(G) \boxtimes A)^\sim) $. More precisely, using that $ \alpha(a) $ and $ \eins \otimes a $ in $ C^\red(G) \otimes A $ for $ a \in A $ correspond to 
the elements $ \eins \boxtimes a $ and $ T(\eins \otimes a) $ in $ C(G) \boxtimes A $, respectively, we get 
\[
M_n((\eins \boxtimes \id_A)^\sim)(p) \oplus \eins_m \oplus 0_l \sim_{\mathrm{MvN}} M_n(\tilde{T}) \circ M_n((\eins \otimes \id_A)^\sim)(p) \oplus \eins_m \oplus 0_l
\]
in $ M_{n + m + l}((C^\red(G) \boxtimes A)^\sim) $. 
Write $ n + m + l = r $. Since $\alpha$ has the Rokhlin property, let
\[
\psi': \big( C^\red(G) \boxtimes A, \Delta\boxtimes\alpha \big) \to \big( A_\infty, \alpha_\infty \big)
\]
be an equivariant $*$-homomorphism with $\psi'(\eins\boxtimes a)=a$ for all $a\in A$. We then consider
\[
\psi = M_r(\tilde{\psi'}) : M_r((C^\red(G) \boxtimes A)^\sim) \rightarrow M_r(\tilde{A})_\infty,
\] 
which is an approximate left-inverse for $ M_r((\eins \boxtimes \id_A)^\sim) $. Then 
\[
p \oplus \eins_m \oplus 0_l \sim_{\mathrm{MvN}} \psi \circ M_n(\tilde{T}) \circ M_n((\eins \otimes \id_A)^\sim)(p) \oplus \eins_m \oplus 0_l
\]
in $ M_r(\tilde{A})_\infty $.
Note that $ M_n((\eins \otimes \id_A)^\sim)(p) $ is contained in the invariant part of 
$ M_n((C^\red(G) \otimes A)^\sim) $. That is, 
\[
M_n((\Delta \otimes \id_A)^\sim) \circ M_n((\eins \otimes \id_A)^\sim)(p) = M_n((\eins \otimes \eins \otimes \id_A)^\sim)(p) 
\] 
in $ M_n((C^\red(G) \otimes C^\red(G) \otimes A)^\sim) $. 
By equivariance of $ T $ and $ \psi $, the same applies to 
\[ 
q_\infty = \psi \circ M_n(\tilde{T}) \circ M_n((\eins \otimes \id_A)^\sim)(p), 
\] 
that is, the latter element 
satisfies 
\[ 
M_n(\tilde{\alpha}_\infty)(q_\infty) = M_n((\eins \otimes \id_{A_\infty})^\sim)(q_\infty) .
\] 
Now the invariant part of $ M_n(\tilde{A})_\infty $ equals $ M_n((A^\alpha)^\sim)_\infty $ by Lemma \ref{fixed-point-alg}. 
Since the relation of being a partial isometry with a fixed range projection is well-known to be weakly stable, this shows that there exists a projection $ q \in M_r((A^\alpha)^\sim) $ such that 
\[
p \oplus \eins_m \oplus 0_l \sim_{\mathrm{MvN}} q
\]
in $ M_r(\tilde{A}) $. Hence 
\[
x = [p] - [\eins_k] = [p \oplus \eins_m] - [\eins_{m + k}] = [q] - [\eins_{m + k}] 
\]
is contained in $ \im(K_0(A^\alpha) \rightarrow K_0(A)) $ as desired.

For the statement about the $K_1$-group, one uses suspension to reduce matters to $K_0$, see the proof of \cite[4.9]{BarlakSzabo15}.
\end{proof}


\section{Examples}

In this final section we present some examples of Rokhlin actions. 

\begin{example} \label{exregularaction}
Let $ G $ be a coamenable compact quantum group acting on $ A = C(G) $ by the regular coaction $ \alpha = \Delta $. Then $ \alpha $ has 
the spatial Rokhlin property.

Indeed, in this case the embedding $ \iota_A: A \rightarrow C(G) \boxtimes A \cong C(G) \otimes A $ is given 
by $ \iota_A = \Delta $. Since $ G $ is coamenable, the counit $ \epsilon: \CO(G) \rightarrow \mathbb{C} $ extends 
continuously to $ C(G) = C^\red(G) $, and $ \id \otimes \epsilon $ is an equivariant left-inverse for $\iota_A$. Hence 
composition with the canonical embedding of $ C(G) = A $ into $ A_\infty $ yields an equivariant approximate left-inverse. 

The $ * $-homomorphism $ \kappa: C(G) \rightarrow A_\infty $ 
corresponding to this Rokhlin action according to Proposition \ref{char:rp} is induced by the canonical embedding of $ C(G) $ into 
its sequence algebra. 
\end{example}

\begin{rem}
Let $G$ be a finite quantum group and $\alpha: A\to C(G)\otimes A$ an action on a separable, unital \cstar-algebra. 
In \cite{KodakaTeruya15}, Kodaka-Teruya introduce and study the Rokhlin property and approximate representability 
in this setting; in fact they also allow for twisted actions in their paper. 
It follows from 
Proposition \ref{char:ar} that $\alpha$ is spatially approximately representable in the sense of Definition \ref{def:ar} if and only if 
it is approximately representable in the sense of Kodaka-Teruya \cite[Section 4]{KodakaTeruya15}. As a consequence, 
Theorem \ref{duality rp ar} shows that $\alpha$ has the spatial Rokhlin property in the sense of Definition \ref{def:rp}, if and only 
if it has the Rokhlin property in the sense of Kodaka-Teruya \cite[Section 5]{KodakaTeruya15}. In particular, our definitions recover 
Kodaka-Teruya's notions of the Rokhlin property and approximate representability and extend them to the non-unital setting. 
A substantial difference between our approach and  \cite{KodakaTeruya15} is that the duality of these two notions becomes 
a theorem rather than a definition.
\end{rem}

\begin{example} \label{fin-qg-UHF}
Let $ G $ be a finite quantum group of order $ n = \dim(C(G)) $. 
Then $ B = M_n=M_n(\mathbb{C}) \cong \mathbb{K}(L^2(G)) $ is a $G$-$ \yd $-\cstar-algebra with the 
coactions $ \beta: B \rightarrow C(G) \otimes B, 
\gamma: B \rightarrow C^*(G) \otimes B $ given by 
\[
\beta(T) = W^*(\eins \otimes T) W, \qquad \gamma(T) = \hat{W}^*(\eins \otimes T) W^*, 
\]
respectively. 
 
Let us write $ B^{\boxtimes k} = B \boxtimes \cdots \boxtimes B $ for the $ k $-fold braided tensor product. 
Note that the embeddings $ M_n^{\boxtimes k} \rightarrow M_n \boxtimes M_n^{\boxtimes k} 
\cong M_n^{\boxtimes k + 1} $ given by $ T \mapsto \eins \boxtimes T $ are $ G $-equivariant. 
As explained in Remark \ref{induct limit coactions}, we may therefore form the inductive limit action $ \alpha: A \rightarrow C(G) \otimes A $ 
of $ G $ on the corresponding inductive limit $ A $. 
 
We remark that $ A $ can be identified with the UHF-algebra $ M_{n^\infty} $. Indeed, the braided tensor product $ M_n \boxtimes M_n $ 
is easily seen to be isomorphic to the ordinary tensor product $ M_n \otimes M_n $ as a \cstar-algebra, using that 
\[
\begin{array}{ccl}
M_n \boxtimes M_n
&=& [\hat{W}^*_{12} (\eins \otimes M_n \otimes \eins) \hat{W}_{12} \beta(M_n)_{13}] \\ 
&\cong& [(M_n \otimes \eins \otimes \eins)(\id \otimes \beta) \beta(M_n)] \\
&\cong& [(M_n \otimes \eins) \beta(M_n)] \\
&\cong& M_n \otimes M_n.
\end{array}
\]
An analogous statement holds for iterated braided tensor products. 

We obtain a $ G $-equivariant $ * $-homomorphism $ \kappa: C(G) \rightarrow A_\infty $ by setting 
$ \kappa(f) = [(\iota_k(f))_{k \in \mathbb{N}}] $, where $ \iota_k: M_n \rightarrow A $ is the embedding into the $ k $-th braided 
tensor factor of $ A $. Moreover, for $ a \in \iota_m(M_n) \subset A $ and $ f \in C(G) $ we have 
\[
a \iota_k(f) = a_{(-2)} \iota_k(f) S(a_{(-1)}) a_{(0)} 
\]
provided $ k > m $. 
It follows that $ \kappa $ satisfies the commutation relations required by Proposition \ref{char:rp}, 
and the norm condition in Proposition \ref{char:rp} $ b) $ is automatic since $ G $ is a finite quantum group. Hence $ (A, \alpha) $ has 
the spatial Rokhlin property. 
\end{example}

\begin{prop} \label{ssaconj}
Let $G$ be a coexact compact quantum group and $\CD$ a strongly self-absorbing \cstar-algebra. Then there exists at most one conjugacy 
class of $G$-actions on $\CD$ with the spatial Rokhlin property.
\end{prop}
\begin{proof}
By \cite[Corollary~1.12]{TomsWinter2007}, any two unital $*$-homomorphisms from $\CD$ to $C^\red(G) \otimes \CD$ are approximately unitarily equivalent. Therefore the claim follows from Theorem \ref{rigidity Rokhlin}.
\end{proof}

\begin{rem}
As a consequence of Proposition \ref{ssaconj} we see that the action of a finite quantum group $G$ of order $n=\dim(C(G))$ 
constructed in Example \ref{fin-qg-UHF} is the unique Rokhlin action of $G$ on $M_{n^\infty}$ up to conjugacy. In particular, it is 
conjugate to the action constructed by Kodaka-Teruya in \cite[Section 7]{KodakaTeruya15}.
\end{rem}

Finally, we shall construct a Rokhlin action of any coamenable compact quantum group on $\CO_2$. As a preparation, 
recall that an element $ a $ in a \cstar-algebra $ A $ is called full if the closed two-sided ideal generated by $ a $ is equal to $ A $. 

\begin{lemma} \label{fullness}
Let $G$ be a coamenable compact quantum group and let $\iota: C(G)\to\CO_2$ be a unital embedding. 
Then $(\id\otimes\iota) \circ \Delta(f)$ is a full element in $C(G)\otimes\CO_2$ for any nonzero $ f \in C(G) $.
\end{lemma}
\begin{proof}
Let $ f \in C(G) = C^\red(G) $ be nonzero. To show that $ h = (\id \otimes \iota)(\Delta(f)) $ is full it is enough to verify 
that $ (\pi \otimes \id)(h) $ is nonzero for all irreducible representations $ \pi $ of $ C(G) $. 
Indeed, if the ideal generated by $ h $ is proper, there must exist a primitive ideal of $ C(G) \otimes \CO_2 $ containing $ h $. 
Since $ \CO_2 $ is nuclear and simple these ideals are of the form $ I \otimes \CO_2 $ for primitive ideals $ I \subset C(G) $, 
see \cite[Theorem 3.3]{Blackadar77}.
Now if $ \pi: C(G) \rightarrow \IL(\CH_\pi) $ is any $ * $-representation, then 
\[ 
(\pi \otimes \id)(\Delta(f)) = (\pi \otimes \id)(W)^*(\eins \otimes f)(\pi \otimes \id)(W) 
\] 
is nonzero in $ \IL(\CH_\pi) \otimes C(G) $, and hence 
\[
(\id \otimes \iota)((\pi \otimes \id)(\Delta(f)) = (\pi \otimes \iota)(\Delta(f)) = (\pi \otimes \id)(h)
\]
is nonzero as well since $ \iota $ is injective.
\end{proof}

\begin{theorem}
Let $G$ be a coamenable compact quantum group. Then up to conjugacy, there exists a unique $G$-action on the Cuntz algebra $\CO_2$ with the spatial Rokhlin property.
\end{theorem}
\begin{proof}
According to Proposition \ref{ssaconj} it suffices to construct some $G$-action on $\CO_2$ with the spatial Rokhlin property. 
Since $C^\red(G)=C(G)$ is nuclear, hence in particular exact, there exists a unital embedding $\iota: C(G)\into\CO_2$. 
For every $n\geq 0$, consider the unital $*$-homomorphism $\Phi_n: C(G)\otimes\CO_2^{\otimes n}\to C(G)\otimes\CO_2^{\otimes n+1}$ given by
\[
\Phi_n(x\otimes y)= \big( (\id\otimes\iota)\circ\Delta \big)(x)\otimes y\quad\text{for all}~x\in C(G),~ y\in\CO_2^{\otimes n}.
\]
Notice that $\Phi_n = \Phi_0\otimes\id_{\CO_2^{\otimes n}}$ for all $n\geq 1$. We have
\[
\begin{array}{ccl}
(\Delta\otimes\id_{\CO_2})\circ\Phi_0 &=& (\Delta\otimes\id_{\CO_2})\circ(\id\otimes\iota)\circ\Delta \\
&=& (\id\otimes\id\otimes\iota)\circ (\Delta\otimes\id)\circ\Delta \\
&=& (\id\otimes\id\otimes\iota)\circ (\id\otimes\Delta)\circ\Delta \\
&=& (\id\otimes\Phi_0)\circ\Delta.
\end{array}
\]
This means that $\Phi_0$ is an injective equivariant $*$-homomorphism
\[
\Phi_0: ( C(G), \Delta) \to ( C(G)\otimes\CO_2, \Delta\otimes\id_{\CO_2} ).
\]
We thus also have that each 
\[
\Phi_n: ( C(G)\otimes\CO_2^{\otimes n} , \Delta\otimes\id_{\CO_2^{\otimes n}} ) \to ( C(G)\otimes\CO_2^{\otimes n+1}, \Delta\otimes\id_{\CO_2^{\otimes n+1}} )
\]
is injective and equivariant. 
Define the inductive limit
\[
(A,\alpha) = \lim_{\longrightarrow}~\set{ ( C(G)\otimes\CO_2^{\otimes n} , \Delta\otimes\id_{\CO_2^{\otimes n}} ), \Phi_n },
\]
where $ \alpha $ denotes the inductive limit coaction, compare Remark \ref{induct limit coactions}. 

Notice that each building block in this inductive limit has the Rokhlin property, and moreover the first-factor embedding of $C(G)$ into 
$C(G)\otimes\CO_2^{\otimes n}$ satisfies the required conditions from Proposition \ref{char:rp} $ b) $ on the nose; 
see Example \ref{exregularaction}. Similarly to what happens in Example \ref{fin-qg-UHF}, the sequence 
\[
\kappa_n=\Phi_{n, \infty}\circ(\id\otimes\eins_{\CO_2^{\otimes n}}): C(G)\to A
\]
yields a $*$-homomorphism $\kappa: C(G)\to A_\infty$ satisfying
the conditions in Proposition \ref{char:rp} $ b) $. Hence $\alpha$ has the Rokhlin property.

By Lemma \ref{fullness}, we know that $\Phi_0$, and thus also each $\Phi_n$ is a full $*$-homomorphism. 
It follows that the inductive limit $A$ is simple, and it is clearly separable, unital and nuclear. 
Moreover $A$ is $\CO_2$-absorbing by \cite[3.4]{TomsWinter2007}. 
This implies $A\cong\CO_2$ due to Kirchberg-Phillips \cite{KirchbergPhillips2000}.
\end{proof}


\bibliographystyle{gabor}
\bibliography{rokhlin}

\end{document}